\documentclass[reqno,11pt]{amsart}
\usepackage{amsmath, amsfonts,amsthm,amssymb,amsbsy,upref,color,graphicx,amscd,hyperref,enumerate}
\usepackage[active]{srcltx}
\usepackage[latin1]{inputenc}
\usepackage{bbm}
\usepackage{pgf}
\usepackage{xfrac}

\def\ds{\displaystyle}
\def\eps{{\varepsilon}}
\def\N{\mathbb{N}}

\def\R{\mathbb{R}}

\def\HH{\mathcal{H}}

\def\KK{\mathcal{K}}

\def\G{\mathcal W_{\sfrac32}}
\def\GG{\mathcal W}

\newcommand{\be}{\begin{equation}}
\newcommand{\ee}{\end{equation}}

\newcommand{\de}{\partial}
\newcommand{\dist}{{\rm {dist}}}

\newcommand{\reg}{{\rm Reg}}
\newcommand{\sing}{{\rm Sing}}
\newcommand{\other}{{\rm Other}}
\newcommand{\freq}{\mathcal{S}}

\theoremstyle{plain}
\newtheorem{theo}{Theorem}

\numberwithin{equation}{section}
\theoremstyle{plain}
\newtheorem{teo}{Theorem}[section]
\newtheorem{lemma}[teo]{Lemma}

\newtheorem{prop}[teo]{Proposition}

\theoremstyle{remark}
\newtheorem{oss}[teo]{Remark}
\newcommand{\ind}{\mathbbm{1}}

\title[Epiperimetric inequalities for the thin obstacle problem]{Direct epiperimetric inequalities for the thin obstacle problem and applications}

\author{Maria Colombo, Luca Spolaor, Bozhidar Velichkov}

\address {Maria Colombo: \newline \indent
Institute for Theoretical Studies, ETH Z\"urich,
	\newline \indent
 Clausiusstrasse 47, CH-8092 Z\"urich, Switzerland
 	}
\email{maria.colombo@eth-its.ethz.ch}

\address {Luca Spolaor: \newline \indent
	Massachusetts Institute of Technology (MIT), 
	\newline \indent
	77 Massachusetts Avenue, Cambridge 
	MA 02139, USA}
\email{lspolaor@mit.edu}

\address {Bozhidar Velichkov: \newline \indent
Laboratoire Jean Kuntzmann (LJK), Universit\'e Grenoble Alpes
\newline \indent
B\^atiment IMAG, 700 Avenue Centrale, 38401 Saint-Martin-d'H\`eres}
\email{bozhidar.velichkov@univ-grenoble-alpes.fr}

\makeindex

\begin{document}



\begin{abstract}
For the thin-obstacle problem, we prove by a new direct method that in any dimension the Weiss' energies with frequency $\sfrac32$ and $2m$, for $m\in \N$, satisfy an epiperimetric inequality, in the latter case of logarithmic type. In particular, at difference from the classical statements, we do not assume any a-priori closeness to a special class of homogeneous function. In dimension $2$, we also prove the epiperimetric inequality at any free boundary point.

As a first application, we improve the set of admissible frequencies for blow ups, previously known to be $\lambda \in \{\sfrac 32\} \cup [2,\infty)$, and we classify the global $\lambda$-homogeneous minimizers, with $\lambda\in [\sfrac32,2+c]\cup\bigcup_{m\in \N}(2m-c_m^-,2m+c_m^+)$, showing as a consequence that the frequencies $\sfrac32$ and $2m$ are isolated. 

Secondly, we give a short and self-contained proof of the regularity of the free boundary previously obtained by Athanasopoulos-Caffarelli-Salsa \cite{AtCaSa} for regular points and Garofalo-Petrosyan \cite{GaPe} for singular points, by means of an epiperimetric inequality of logarithmic type which applies for the first time also at all singular points of thin-obstacle free boundaries. In particular we improve the $C^1$ regularity of the singular set with frequency $2m$ by an explicit logarithmic modulus of continuity.

\end{abstract}

\maketitle






\section{Introduction}


In this paper we study the regular and singular parts of the free-boundary for solutions of the \emph{thin-obstacle problem}, that is the minimizers of the Dirichlet energy
$$
\mathcal{E}(u):=\int_{B_1}|\nabla u|^2\,dx
$$
in the class of admissible functions
$$
\mathcal A:=\big\{ u\in H^1(B_1)\,:\,u\geq 0 \mbox{ on }B_1'   \,,\,\,u(x',x_d)=u(x',-x_d) \mbox{ for every }(x',x_d)\in B_1 \big\}\,,
$$
with Dirichlet boundary conditions $u=w$ on $\de B_1 $.
Here and in the rest of the paper $d\geq 2$, $B_1\subset \R^d$ denotes the unit ball, $B_1' = B_1 \cap \{ x_{d}=0 \}$,
 for any $x=(x_1,\dots,x_d)\in\R^d$ we denote by $x'$ the vector of the first $d-1$ coordinates, $x'=(x_1,\dots,x_{d-1})$, and $w \in \mathcal A$ is a given boundary datum. 

\noindent Given a minimizer $u\in \mathcal A$ of $\mathcal E$ with Dirichlet boundary conditions the \emph{coincidence set} $\Delta(u)\subset B_1'$ is defined as $\Delta(u):=\{(x',0)\in B_1'\,:\, u(x',0)=0\}$ and the \emph{free boundary $\Gamma(u)$ of $u$} is the topological boundary of the coincidence set in the relative topology of $B_1'$. 
\subsection{State of the art} Athanasopoulos and Caffarelli \cite{atcaffa} proved that the optimal regularity of any local minimizer $u$ is $C^{1,1/2}(B_1^+)$. Athanasopoulos, Caffarelli and Salsa pioneered the study of the regularity of the free boundary $\Gamma(u)$ in \cite{AtCaSa}. They showed in \cite[Lemma 1]{AtCaSa} that for every $x_0\in \Gamma(u)$ the \emph{Almgren's frequency function}
$$
(0,1-|x_0|) \ni r\mapsto N^{x_0}(r,u):=\frac{r\int_{B_r(x_0)}|\nabla u|^2\,dx}{\int_{\de B_r(x_0)}u^2\,d\HH^{d-1}}\,
$$ 
is monotone nondecreasing in $r$. Thus, the limit 
$$N^{x_0}(0,u):=\lim_{r\to 0}N^{x_0}(r,u)$$
exists for every point $x_0\in\Gamma(u)$ and the free boundary can be decomposed according to the value of the frequency function in zero. We denote the set of points of frequency $\lambda\in\R$ by  
$$\freq^{\lambda}(u):=\{x\in \Gamma(u)\,:\,N^x(0,u)=\lambda\}.$$
Using the frequency function one can split the free-boundary into three disjoint sets 
\begin{itemize}
\item the \emph{regular free boundary} which consists of the points with the lowest possible frequency
$$
\reg(u):= \freq^{\sfrac32}(u)\,;
$$
\item the points with even integer frequency $\freq^{2m}(u)$, whose union  by definition constitutes the set of \emph{singular points} $\sing(u)$
$$\sing(u):=\bigcup_{m\in \N} \freq^{2m}(u);$$
\item the remaining part, denoted in the literature by $\other(u)$.
\end{itemize}

The first result on the regularity of the free boundary for the thin-obstacle problem is due to Athanasopoulos, Caffarelli and Salsa. In \cite{AtCaSa} they give a complete description of the blow-up limits at the points of frequency $\sfrac32$ and prove that the regular free boundary $\reg(u)$ is locally a $(d-2)$-dimensional $C^{1,\alpha}$ hypersurface in $\R^{d-1}$. Later the regular part of the free boundary has been shown to be $C^\infty$ in \cite{yashrobin, korush} and analogous results were extended to more general fractional laplacian (see \cite{CaSaSi}), of which the thin-obstacle is a particular example.

Garofalo and Petrosyan (cp. \cite[Theorem 2.6.2]{GaPe}) showed that $\sing(u)$ is precisely the set of points where the coincidence set is asymptotically negligible, that is
\begin{equation}
\label{eqn:GP-dens0}
\sing(u)=\left\{x_0\in\Gamma(u)\,:\,\lim_{r\to 0}\frac{\HH^{d-1}(\Delta(u)\cap B'_r(x_0))}{\HH^{d-1}(B'_r(x_0))}=0\right\}\,.
\end{equation}
With the help of new monotonicity formulas of Weiss and Monneau type, Garofalo and Petrosyan showed that each set $\mathcal{S}^{2m}$ is contained in a countable union of $C^1$ manifolds in $\R^{d-1}$.

In general the set $\other(u)$ is not empty and is not even small compared to the free boundary $\Gamma(u)$. Indeed, in dimension two the function $h(r,\theta)=r^{2m-\sfrac12}\sin\left(\frac{1-4m}{2}\theta\right)$ is a global solution with frequency $2m-\sfrac12$ in zero. Using this example one can easily construct global solutions in any dimension $d\ge 2$ whose entire free-boundary is a $(d-2)$-dimensional plane consisting only of points with frequency $2m-\sfrac12$. 
Recently, Focardi and Spadaro \cite{FoSp-GMT} proved the $\HH^{d-2}$-rectifiability of the set $\other(u)$ and that it consists of points of frequency $2m-\sfrac12$ up to a set of zero $\HH^{d-2}$ measure, but nothing is known up to now regarding its regularity in dimension $d>2$. We notice that in some special cases, the set $\other(u)$ might be empty. Indeed, Barrios, Figalli and Ros-Oton proved in \cite{barriosfiga} that this is precisely the case when the admissibility condition $u\ge 0$ is replaced by $u\ge \varphi$ on $\R^{d-1}$, where $\varphi$ is a non-zero superharmonic obstacle. 

A different approach for the regularity of the free boundary was proposed by Garofalo-Petrosyan-Vega-Garcia \cite{GaPeVe} and Focardi-Spadaro \cite{FoSp}, following the result of Weiss \cite{Weiss2} for the classical obstacle problem. For points of the regular free boundary $x_0\in \reg(u)=\freq^{\sfrac32}$, they prove an epiperimetric inequality for the \emph{Weiss' boundary adjusted energy}
$$\GG_{\lambda}^{x_0}(r,u):=\frac1{r^{d-2+2\lambda}}\int_{B_r(x_0)}|\nabla u|^2\,dx-\frac{\lambda}{r^{d-1+2\lambda}}\int_{\de B_r(x_0)}u^2\,d\HH^{d-1}\,,$$
which allows to quantify the convergence of $\GG_{\lambda}^{x_0}(r,u)$ as $r\to 0$ to be of H\"older type and provides an alternative proof of the $C^{1,\alpha}$ regularity of the free boundary. The epiperimetric inequality approach was first introduced by Reifenberg \cite{Reif2}, White \cite{Wh} and Taylor \cite{taylor} in the context of minimal surfaces, later brought to the classical obstacle problem by Weiss \cite{Weiss2} and recently developed in \cite{SpVe} with new contributions in the framework of free boundaries.

\subsection{Main results} In this paper we present a new interpretation of the epiperimetric inequality not as a property of the energy and its homogeneous minimizers, but as a property of the family of Weiss' boundary adjusted energies $\GG_{\lambda}$, $\lambda=\sfrac32,2m$. Indeed this approach \emph{doesn't require any a-priori knowledge of the admissible blow-ups} (which even in the previous results \cite{GaPeVe, FoSp} about the regular points was assumed, by requiring a suitable closeness to the already-known blow up), and actually yields their classification. Moreover, as usual, it gives a short, self-contained proof of the known regularity of $\reg(u)$
and, thanks to the direct arguments at the basis of the epiperimetric inequality, allows to obtain a new logarithmic modulus of continuity for the singular set, which improves the results of \cite{FoSp,GaPeVe}.

\subsubsection{Epiperimetric inequalities for $\GG_\lambda$, $\lambda=\sfrac32, 2m$, in any dimension}

In this section we present our epiperimetric inequalities. Notice that, at difference from the existing literature, they hold for \emph{any trace $c$} without any closeness assumption to the admissible blow ups.
 \smallskip
For the energy $\GG_{\sfrac32}$ we give a short and self-contained proof of the following statement.
\begin{theo}[Epiperimetric inequality for $\GG_{\sfrac32}$]\label{t:epi:3/2}
Let $d\ge 2$ and $B_1\subset \R^{d}$. 
Then for every $c\in H^1(\partial B_1)$ such that its $3/2$-homogeneous extension $z(r,\theta):=r^{3/2}c(\theta)$ belongs to $\mathcal A$, there exists $v\in \mathcal A$ such that $v=c$ on $\partial B_1$ and 
	\begin{equation}\label{e:epi:3/2}
	\GG_{\sfrac32}(v)\le \left(1-\frac{1}{2d+3}\right)\GG_{\sfrac32}(z).
	\end{equation}
\end{theo}
\noindent A similar statement was obtained in \cite{GaPeVe, FoSp}, even though in these papers a further assumption is required (the closeness of the boundary datum $c$ to the set of admissible blow ups of frequency $\sfrac32$) and it is based on a contradiction argument. The proof of Theorem~\ref{t:epi:3/2} exhibits instead an explicit energy competitor $v$, after decomposing the boundary datum $c$ in terms of the eigenfunctions of the Laplacian on $\partial B_1$. Roughly speaking, in $v$ we extend with homogeneity $\alpha>\sfrac32$ the high modes on the sphere, whereas the rest is extended with the same homogeneity as $z$.
This line of proof was followed by the authors for the classical obstacle problem in \cite{cospve-classicobst} and by the last two named authors for the Alt-Caffarelli functional in dimension $2$ in~\cite{SpVe}.\\

In analogy to the results on the classical obstacle problem \cite{cospve-classicobst}, our direct approach allows to obtain a logarithmic epiperimetric inequality for the family of energies $\GG_{2m}$, $m\in \N$, in any dimension. This, together with \cite{cospve-classicobst}, is the first instance in the literature (even in the context of minimal surfaces) of an epiperimetric inequality of logarithmic type, and the first instance in the context of the lower dimensional obstacle problems where an epiperimetric inequality for singular points has a direct proof. This result allows us to prove a complete and self-contained regularity result for $\sing(u)$ and improve the known results by giving an explicit modulus of continuity. Further applications to other singular points of the thin obstacle problem and to the fractional obstacle problem for any $s\in (0,1)$ will be presented in \cite{CoSpVe-prep,future}.

\begin{theo}[Logarithmic epiperimetric inequality for $\GG_{2m}$]\label{t:epi:2m} Let $d\ge 2$, $m\in \N$.
For every function $c\in H^1(\partial B_1)$ such that its $2m$-homogeneous extension $z(r,\theta)=r^{2m}c(\theta)$ is in $\mathcal A$ and 
\begin{equation}\label{e:orchecattive}
\int_{\partial B_1}c^2\,d\HH^{d-1}\le 1\qquad\text{and}\qquad |\GG_{2m}(z)|\le 1\,,
\end{equation}
there are a constant $\eps=\eps(d,m)>0$ and a function $h\in \mathcal A$, with $h=c$ on $\partial B_1$, satisfying 
	\begin{equation}\label{e:epi:2m}
	\GG_{2m}(h)\le \GG_{2m}(z) \big(1-\eps \,|\GG_{2m}(z)|^{\gamma}\big),\qquad \mbox{where}\quad \gamma:=\frac{d-2}{d}\,.
	\end{equation}
\end{theo}

It is not difficult to see that with our method the power $0<\gamma<1$ in \eqref{e:epi:2m} cannot be avoided, see for instance \cite[Example 1]{cospve-classicobst}. This is essentially due to the possible convergence of polynomial of fixed degree $2m$ with low symmetry to ones with higher symmetry.  

\subsubsection{Complete analysis of the free boundary points in dimension two}
In dimension $d=2$, it is known that the only admissible values of the frequency at points of the free boundary are $\sfrac32, 2m$ and $2m-\frac12$, for $m\in \N$. Theorem \ref{t:epi:3/2} and Theorem \ref{t:epi:2m} already provide the classical epiperimetric inequalty for the points $\sfrac32$ and $2m$; indeed, in the case $d=2$, we have $\gamma=0$ in \eqref{e:epi:2m}. We complete the analysis in dimension two by proving an epiperimetric inequality also at the points of density $2m-\frac12$. Before we state the theorem, we recall that in this case the admissible blow up is (up to a constant and a change of orientation) of the form
$$h_{2m-\sfrac 12}(r,\theta)=r^{\frac{4m-1}{2}}\sin\left(\frac{1-4m}{2}\theta\right).$$

\begin{theo}[Epiperimetric inequality for points of frequency $2m-\sfrac12$ in dimension two]\label{t:epi:2m-1/2}
	Let $d=2$ and  $m \in \N$. There exist constants $\delta>0$ and $\kappa>0$ such that the following claim holds. For every function $c\in H^1(\partial B_1)$ such that its $2m-\sfrac12$ homogeneous extension $z\in \mathcal A$  and satisfying
	\begin{equation}
	\label{hp:vicino}
	\|c-h_{2m-\sfrac 12}\|_{L^2(\partial B_1)}\le \delta\,,
	\end{equation}
	there exists $h\in \mathcal A$ such that $h|_{\de B_1}=c$ and 
	\begin{equation}\label{e:epi:2m-1/2}
	\GG_{2m-\sfrac12}(h) \le (1-\kappa)\GG_{2m-\sfrac12}(z).
	\end{equation}
\end{theo}
\noindent In dimension $d=2$, the regularity of the free boundary (namely, the fact that they are isolated in the line) can be obtained also with softer arguments than our epiperimetric inequality; however, the previous result allows for instance to show the $C^{1,\alpha}$ decay of $u$ on the unique blow up at each free boundary point and also provides an alternative, self-contained approach.

\subsubsection{Application of the epiperimetric inequalities I: homogeneous minimizers and admissible frequencies} A very important and not yet well-understood question in the contest of the thin-obstacle problem is the study of the admissible frequencies at free-boundary points. Indeed nothing is known, except for the gap between $\sfrac32$ and $2$ (see \cite{AtCaSa}) and the recent result of Focardi and Spadaro \cite{FoSp-GMT}, where they establish that the collection of free-boundary points with frequency different than $\sfrac32$, $2m$ and $2m-\sfrac12$, is a set of $\HH^{d-2}$ measure zero. It is conjectured that these are the only admissible frequencies, but not even the gap between $2$ and the subsequent admissible frequency was known. Indeed, thanks to Theorem \ref{t:freq} below, we are able to recover the gap $\sfrac32-2$ and to prove the new result that \emph{the frequencies $2m$ are isolated for every $m\in \N$}, where the gap is given by explicit constants.

We say that $\lambda\in\R$ is an {\it admissible frequency} if there is a solution $u\in H^1(B_1)$ of the thin-obstacle problem and a point $x_0\in\Gamma(u)$ such that $N^{x_0}(0)=\lambda$. For a minimizer $u$ and an admissible frequency $\lambda=N^{x_0}(0)$, the monotonicity of the frequency function implies that, up to a subsequence, $\|u_{r,x_0}\|_{L^2(\partial B_1)}^{-1}{u_{r,x_0}}$ converges, as $r\to 0$, weakly in $H^1(B_1)$ and strongly in $L^2(B_1)\cap L^2(\partial B_1)$ to a $\lambda$-homogeneous global solution $p:\R^d\to\R$ such that $\|p\|_{L^2(\partial B_1)}=1$. In particular, if we denote by 
$$
\KK_{\lambda} := \{u\in H^1(B_1): \mbox{$u$ is a nonzero $\lambda$-homogeneous minimizer of $\mathcal E$ and $u \geq 0 $ on $B_1'$}\}
$$
we have that 
\begin{align}\label{e:bu}
\mbox{if $\lambda$ is an admissible frequency, then $\KK_{\lambda} \neq \emptyset$}.
\end{align}

A complete description of the spaces $\KK_\lambda$ and the admissible frequencies is known only in dimension two, where the only possible values of $\lambda$ are $\sfrac32,\, 2m,$ and $2m-\sfrac12$ for $m\in \N_+$.  However, as a consequence of our logarithmic epiperimetric inequality we can describe the set $\KK_\lambda$ for values of $\lambda$ close to $2m$.

\begin{theo}[$\lambda$-homogeneous minimizers]\label{t:freq}
Let $d \geq 2$. Then for every $m\in \N$ there exist constants $c_m^\pm>0$, depending only on $d$ and $m$, such that 
\begin{equation}
\label{ts:gap}
\KK_{\lambda} = \emptyset \qquad 
\mbox{for every }\lambda \in (\sfrac32,2)\cup \bigcup_{m\in \N} \big((2m-c_m^-,2m) \cup(2m,2m+c_m^+)\big).
\end{equation}
Moreover, setting
\begin{equation}
\label{eqn:he}
h_e(x):=\left(2( x'\cdot e)-\sqrt{( x'\cdot e)^2+x_d^2}\right)\sqrt{\sqrt{( x'\cdot e)^2+x_d^2}+ x\cdot e} = {\rm Re}( x'\cdot e + i |x_d|)^{\sfrac32}\ ,
\end{equation}
we have
\begin{equation}\label{e:class:3/2}
\KK_{\sfrac32}=\{C\, h_e\,:\,e\in \mathbb S^{d-1}\quad\mbox{and}\quad C>0\}\,,\end{equation}
\begin{equation}\label{e:class:2m}
\begin{array}{ll}
\KK_{2m}=\{C\, p_{2m}\,:\,p_{2m}&\text{is a $2m$-homogeneous harmonic polynomial},\\
& \quad p_{2m}\ge 0\text{ on }B_1',\quad \|p_{2m}\|_{L^2(\partial B_1)}=1\quad\text{and}\quad C>0\}\,.
\end{array}
\end{equation}
\end{theo}

\begin{oss} Theorem \ref{t:freq} and \eqref{e:bu} imply that the frequencies $\sfrac32$ and $2m$, for every $m\in \N$, are isolated, and in particular $N^{x_0}\notin (\sfrac32,2)\cup \bigcup_{m\in \N} \big((2m-c^-_m,2m) \cup(2m+c^+_m)\big)$ for every $x_0\in \Gamma(u)$, where $u$ is a minimizer of the obstacle problem for general obstacle $\phi$.	

At difference with respect to other results where gaps of this kind are established, the arguments leading to the constants $c_m$ are never by contradiction, hence the constants $c_m$ can be tracked in the proofs (see Remark~\ref{rmk:explicit} for an explicit example).
\end{oss}

We wish to stress that the classes $\KK_{2m}$ and $\KK_{\sfrac32}$ were already characterized (see \cite{AtCaSa,GaPe}) and that typically this characterization is needed to prove an epiperimetric inequality. However our epiperimetric inequalities are a property of the energies $\GG_{\lambda}$, and not of a class of blow-ups, and as such allow us to characterize the $\KK_{\lambda}$ as a corollary.

\begin{oss}
	Finally we notice that \eqref{eqn:GP-dens0} follows immediately from Theorem \ref{t:freq}, the classification, thus giving an alternative proof to the one of \cite{GaPe}.
\end{oss}

\subsubsection{Application of the epiperimetric inequalities II: regularity of the free boundary in any dimension}

Using the epiperimetric inequalities Theorem \ref{t:epi:3/2} and Theorem \ref{t:epi:2m} we prove the following regularity result, valid in any dimension.
\begin{theo}[Regularity of the Regular and Singular set]\label{t:main}
	Let $u\in \mathcal{A}_w$ be a minimizer of the thin-obstacle energy $\mathcal E$.
	\begin{itemize}
		\item[(i)] There exists a dimensional constant $\alpha>0$ such that $\reg(u)$ is in $B_1'$ a $C^{1,\alpha}$ regular open submanifold of dimension $(d-2)$.
		\item[(ii)] For every $m\in \N$ and $k=0,\dots, d-2$, $S_k^{2m}(u)$ is contained in the union of countably many submanifolds of dimension $k$ and class $C^{1,log}$. In particular $\sing(u)$ is contained in the union of countably many submanifolds of dimension $(d-2)$ and class $C^{1,log}$.
	\end{itemize}	
\end{theo}

\begin{oss} If we consider minimizers $u\in H^1(B_1^+)$ with Dirichlet boundary conditions of the more general thin-obstacle problem, where we minimize the energy $\mathcal E$ in the class of admissible functions
	$$
	\mathcal A^\phi:=\{ u\in H^1(B_1^+)\,:\,u\geq \phi \mbox{ on }B_1'   \,,\,\,u(x',x_d)=u(x',-x_d) \mbox{ for every }(x',x_d)\in B_1  \}\,,
	$$
	with $\phi\in C^{l,\beta}(B_1',\R+)$, then an analogous statement holds, that is 
	\begin{itemize}
		\item[(i)] there exists a dimensional constant $0<\alpha\leq \beta$ such that $\reg(u)$ is in $B_1'$ a $C^{1,\alpha}$ regular submanifold of dimension $(d-2)$,
		\item[(ii)] for every $2m<l$ and $k=0,\dots, d-2$, $S_k^{2m}(u)$ is contained in the union of countably many submanifolds of dimension $k$ and class $C^{1,log}$.
	\end{itemize}	
	This result can be proved as a standard application of our various epiperimetric inequalities and the almost minimality of the blow-ups at a point of the free-boundary, which follows from the regularity of the obstacle (see for instance \cite{cospve-classicobst}). In particular it provides an improvement in the regularity of $\freq^{2m}$, $2m<l$, from $C^1$ to $C^{1,\log}$ of the results of \cite{GaPe, barriosfiga}.  
\end{oss}


\subsection{Organization of the paper}The paper is organized as follows. After introducing notation and classical results in Section~\ref{sec:prelim}, Sections~\ref{sec:epi-reg}, \ref{sec:epi-sing}, and \ref{sec:epi-2d} are devoted to the proofs of the epiperimetric inequalities of Theorems~\ref{t:epi:3/2}, \ref{t:epi:2m} and \ref{t:epi:2m-1/2}, respectively. 
 Section \ref{sec:freq} contains the proof of Theorem \ref{t:freq}, which is new and follows from our direct approach to the epiperimetric inequality. 
Section~\ref{sec:regFB} is dedicated to the proof of Theorem \ref{t:main} which is based on arguments of classical flavor and which is adapted to the logarithmic case.

%
%


\section{Preliminaries}\label{sec:prelim}

In this section we recall some properties of the solutions of the thin-obstacle problem, the frequency function, the Weiss' boundary adjusted functional and we deal with some preliminary computations. 

\subsection{Regularity of minimizers} The optimal regularity of the solutions of the thin obstacle problem was proved in \cite{atcaffa}. We recall the precise estimate in the following theorem.

\begin{teo}[Optimal regularity of minimizers {\cite{atcaffa}}]\label{t:ACS}	Let $u\in \mathcal A$ be a minimizer of $\mathcal E$ with Dirichlet boundary conditions. Then $u\in C^{1,\sfrac12}(B_{\sfrac12}^+)$ and there exists a dimensional constant $C_d>0$ such that 
	$$
	\|u\|_{C^{1,\sfrac12}(B_{\sfrac12}^+)}\leq C_d\, \|u\|_{L^2(B_1)}\,.
	$$
\end{teo}

\subsection{Properties of the frequency function} Let $u\in H^1(B_1)$ be a minimizer of the thin-obstacle energy and $x_0\in \Gamma(u)$. Then we introduce the quantities
$$
D^{x_0}(r):=\int_{B_r(x_0)}|\nabla u|^2\,dx,\quad H^{x_0}(r):=\int_{\de B_r(x_0)}u^2\,d\HH^{d-1}
\quad\mbox{and}\quad N^{x_0}(r):=\frac{r\, D^{x_0}(r)}{H^{x_0}(r)}\,,
$$
where $0<r<1-|x_0|$. Furthermore in this notation we have
$$
\GG^{x_0}_{\lambda}(r,u)=\frac1{r^{d-2+2\lambda}}D^{x_0}(r)-\frac{\lambda}{r^{d-1+2\lambda}}H^{x_0}(r):=\GG^{x_0}_{\lambda}(r)\,.
$$

In the following we will need the monotonicity of $N$, which can be found in \cite{AtCaSa}, and of $\GG_\lambda$, which can be found in \cite{FoSp} in the case of frequency $\sfrac32$. For the sake of completeness we give here a proof in the general case.
\begin{lemma}[Properties of the frequency function]\label{lem:freq}
	Let $u\in H^1(B_1)$ be a minimizer of $\mathcal E$ and $x_0\in \Gamma(u)$, then the following properties hold. 
	\begin{itemize}
		\item The functions $N^{x_0}(r)$ and $\GG_\lambda^{x_0}(r)$, for any $\lambda>0$, are monotone nondecreasing and in particular
		\begin{equation}\label{e:monot}
		\frac{d}{dr}\GG_{\lambda}^{x_0}(r)=\frac{(d-2+2\lambda)}{r}\left(\GG_\lambda(z_r)-\GG_\lambda(u_r) \right)+\frac1r\int_{\de B_1} \left( \nabla u_r\cdot \nu-\lambda\,u_r \right)^2\,d\HH^{d-1}\,,
		\end{equation}
		where $\ds u_r(x):=\frac{u(x_0+rx)}{r^\lambda}$ and $\ds z_r(x):=|x|^\lambda\, u_r\left(\sfrac{x}{|x|}\right)$. 
		\item For every $N^{x_0}(0)>\lambda$, the function $\ds \frac{H^{x_0}(r)}{r^{d-1+2\lambda}}$ is monotone nondecreasing and in particular
		\begin{equation}\label{e:der_H}
		\frac{d}{dr}\left(\frac{H^{x_0}(r)}{r^{d-1+2\lambda}} \right)=2\frac{\GG_\lambda^{x_0}(r)}{r}\,.
		\end{equation}
	\end{itemize}
\end{lemma} 

\begin{proof}
	For the monotonicity of $\GG_{\lambda}$, dropping the index $x_0$, we recall the identities 
	\begin{gather}
	D'(r)=(d-2)\,\frac{D(r)}{r}+2\,\int_{\de B_r} (\de_\nu u)^2\,d\HH^{d-1} \label{e:D'}\\
	H'(r)=(d-1)\,\frac{H(r)}{r}+2\,\int_{\de B_r} u\,\de_\nu u\,d\HH^{d-1} \label{e:H'}\\
	D(r)=\int_{\de B_r} u\,\de_\nu u\,d\HH^{d-1}\notag\,.\label{e:D}
	\end{gather}
	Then, similarly to \cite{FoSp}, we compute
	\begin{align}\label{e:mon_1}
	\GG_{\lambda}'(r)	
		&=\frac{D'(r)}{r^{d-2+2\lambda}}-(d-2+2\lambda)\,\frac{D(r)}{r^{d-1+2\lambda}}-\lambda\,\frac{H'(r)}{r^{d-1+2\lambda}}+\lambda\,(d-1+2\lambda)\,\frac{H(r)}{r^{d-1+2\lambda}}\notag\\
		&\stackrel{\eqref{e:H'}}{=}-\frac{(d-2+2\lambda)}{r}\,\GG_{\lambda}(r) -\lambda\,(d-2+2\lambda)\,\frac{H(r)}{r^{d+2\lambda}} +\underbrace{\frac{D'(r)}{r^{d-2+2\lambda}}+2\lambda^2\,\frac{H(r)}{r^{d+2\lambda}}-2\lambda\,\frac{D(r)}{r^{d-1+2\lambda}}}_{=:I(r)}\,.
	\end{align}
	Next a simple computation shows that
	\begin{align*}
	I(r)
		&=\frac1r\int_{\de B_1} \left(|\nabla u_r|^2-2\lambda\,u_r\,\de_\nu u_r+2\lambda^2\,u_r^2\right)\,d\HH^{d-1}\\
		&=\frac1r\int_{\de B_1} \left[\left(|\de_\nu u_r|-\lambda\,u_r\right)^2+|\nabla_\theta u_r|^2+\lambda^2\,u_r^2\right]\,d\HH^{d-1}\\
		&=\frac1r\int_{\de B_1} \left(|\de_\nu u_r|-\lambda\,u_r\right)^2\,d\HH^{d-1}+(d-2+2\lambda)\,\int_{B_1}|\nabla z_r|^2
	\end{align*}
	which, together with \eqref{e:mon_1}, implies
	$$
	\GG'_{\lambda}(r)=\frac{(d-2+2\lambda)}{r}\left(\GG_\lambda(z_r)-\GG_\lambda(u_r) \right)+\frac1r\int_{\de B_1} \left( \nabla u_r\cdot \nu-\lambda\,u_r \right)^2\,d\HH^{d-1}\,.
	$$
	In particular, if $u$ minimizes $\mathcal E$, then the monotonicity of $\GG_\lambda$ follows.	
	
	For the second bullet, we can compute
	\begin{align*}
	\frac{d}{dr}\left(\frac{H(r)}{r^{d-1+2\lambda}} \right)
		&=\frac{H'(r)}{r^{d-1+2\lambda}}-	(d-1+2\lambda)\,\frac{H(r)}{r^{d+2\lambda}}\\
		&\stackrel{\eqref{e:H'}}{=}(d-1)\frac{H(r)}{r^{d-2+2\lambda}}+\frac{2}{r^{d-1+2\lambda}}\,\int_{\de B_r} u\,\de_\nu u\,d\HH^{d-1}-(d-1+2\lambda)\,\frac{H(r)}{r^{d+2\lambda}}\\
		&\stackrel{\eqref{e:D}}{=}2\,\frac{D(r)}{r^{d-1+2\lambda}}-(2\lambda)\,\frac{H(r)}{r^{d+2\lambda}}=\frac{2}{r}\GG_\lambda(r)\,.
	\end{align*}
	Notice that $\ds \GG_{\lambda}(r)=\frac{H(r)}{r^{d-1+2\lambda}}(N(r)-\lambda)$, so that if $N(0)>\lambda$, then $\GG_\lambda(r)$ is positive, by monotonicity of $N(r)$, and the claim follows. 
\end{proof}

\subsection{Blow-up sequences, blow-up limits and admissible frequencies}
Given a function $u\in H^1(B_1)$ minimizing the energy $\mathcal E$ and a point $x_0\in \mathcal{S}^\lambda$, we define the \emph{blow-up sequence of $u$ at $x_0$} by  $u_{x_0,r}(x):=\frac{u(x_0+rx)}{r^\lambda}$. Using the monotonicity of $N^{x_0}$ and $H^{x_0}$ it is easy to see that
$$
\int_{B_1}|\nabla u_{x_0,r}|^2\,dx= \frac{1}{r^{d-2+2\lambda}}\int_{B_r(x_0)}|\nabla u|^2\,dx= N^{x_0}(r)\, \frac{H^{x_0}(r)}{r^{d-1+2\lambda}}\leq N^{x_0}(1)\, H^{x_0}(1)\,.
$$ 
It follows that there exists a subsequence $(u_{x_0,r_k})_k$ and a function $u_{x_0}$, which depends on the subsequence, such that $u_{x_0,r_k}$ converges weakly in $H^1(B_1)$ and strongly in $L^2(B_1)\cap L^{2}(\de B_1)$ to some function $p_{x_0}\in H^1(B_1)$. Furthermore by Theorem \ref{t:ACS} we have that the convergence is in $C^{1,\alpha}_{loc}(B_1)$, for every $\alpha<\sfrac12$, and by the minimality of $u$, it is also strong in $H^1(B_1)$. A standard argument using the monotonicity of $\GG_{\lambda}^{x_0}$ then shows that $p_{x_0}$ is a $\lambda$-homogeneous global minimizer of $\mathcal E$ such that $p_{x_0}(x',0)\geq0$. We say that $p_{x_0}$ is a blow-up limit at $x_0$ and we denote by $\KK^{x_0}(u)$ the set of all possible blow-up limits at $x_0$.


\subsection{Fourier expansion of the Weiss' energy}\label{sub:fourier}
On the $(d-1)$-dimensional sphere $\partial B_1\subset \R^d$ we consider the Laplace-Beltrami operator $\Delta_{\partial B_1}$. Recall that the spectrum of $\Delta_{\partial B_1}$ is discrete and is given by the decreasing sequence of eigenvalues (counted with the multiplicity)
$$0= \lambda_1\le \lambda_2\le \dots\le \lambda_k\le \dots$$
The corresponding normalized eigenfunctions $\phi_k:\partial B_1\to \R$ are the solutions of the PDEs
$$-\Delta_{\partial B_1}\phi_k=\lambda_k\phi_k\quad\text{on}\quad \partial B_1,\qquad\int_{\partial B_1}\phi_k^2\,d\HH^{d-1}=1.$$
For every $\mu\in\R$ we will use the notation 
\begin{equation}
\label{defn:lambda}
\lambda(\mu)=\mu(\mu+d-2),
\end{equation}
and we will denote by $\alpha_k$ the unique positive real number such that 
$\lambda(\alpha_k)=\lambda_k.$
It is easy to check that the homogeneous function $u_k(r,\theta)=r^{\alpha_k}\phi_k(\theta)$ is harmonic in $\R^d$ if and only if its trace $\phi_k$ is an eigenfunction on the sphere corresponding to the eigenvalue $\lambda_k$. Moreover, it is well known that in any dimension the homogeneities $\alpha_k$ are natural numbers and the functions $u_k$ are harmonic polynomials of homogeneity $\alpha_k$. 
Furthermore for every $\lambda\geq0$ eigenvalue of the Laplace-Beltrami operator on the sphere, we define 
$$ E(\lambda):=\left\{\phi \in H^1(\de B_1)\,:\,-\Delta_{\de B_1}\phi=\lambda\, \phi\quad\mbox{and}\quad\|\phi\|_{L^2(\de B_1)}\neq0\,\right\}\,,$$
that is $E(\lambda)$ is the eigenspace of $\Delta_{\de B_1}$ associated to the eigenvalue $\lambda$ intersected with the unit sphere. We write the energy of a homogeneous function in terms of its Fourier coefficients; a similar lemma can be found in \cite[Lemma 2.1]{cospve-classicobst}, but we report the short proof for completeness.
\begin{lemma}\label{l:thin_fourier}
Let $d \geq 2$, $\alpha, \mu>0$ and 
\begin{equation}\label{eps_choice}
	\kappa_{\alpha,\mu}:=\frac{\alpha-\mu}{\alpha+\mu+d-2}\,.
	\end{equation}
With the notations above, let $\psi=\sum_{j=1}^\infty c_j\phi_j\in H^1(\partial B_1)$, let $\varphi_\alpha (r,\theta):=r^\alpha\psi(\theta)$ the $\alpha$-homogeneous extension of $\psi$ in $B_1$.
Then we have
\begin{equation}\label{e:W_m_vero}
	\GG_{\mu} (\varphi_\mu)= \frac1{2\mu+d-2}\sum_{j=1}^\infty \left(\lambda_j-\lambda(\mu)\right)c_j^2\,,
	\end{equation} 
	\begin{equation}\label{e:W_m}
	\GG_\mu(\varphi_\alpha)-(1-\kappa_{\alpha,\mu})\GG_\mu(\varphi_\mu)= \frac{\kappa_{\alpha,\mu}}{d+2\alpha-2}\sum_{j=1}^\infty (-\lambda_j+\lambda(\alpha)) c_j^2\,.
	\end{equation} 
\end{lemma}

\begin{proof}
Since $\| \varphi_j\|_{L^2(\partial B_1)} =1$ and $\|\nabla_{\theta} \varphi_j\|_{L^2(\partial B_1)} =\lambda_j$ for every $j \in \{0\} \cup \N$,  we have
	\begin{align*}
	\GG_\mu( \varphi_\alpha)&= \sum_{j=1}^\infty c_j^2\left(\int_0^1 r^{d-1}\,dr\int_{\partial B_1}\!\!\!\,d\HH^{d-1}\left[\alpha^2r^{2\alpha-2}\phi_j^2(\theta)+r^{2\alpha-2}|\nabla_\theta\phi_j|^2(\theta)\right]-\mu\int_{\partial B_1}\!\!\!\phi_j^2(\theta)\,d\HH^{d-1}\right)\\
	&= \sum_{j=1}^\infty c_j^2\left(\frac{\alpha^2+\lambda_j}{d+2\alpha-2}-\mu\right),
	\end{align*} 
	where in the above identity $d\theta$ stands for the Hausdorff measure $\HH^{d-1}$ on the sphere $\partial B_1$. When $\alpha=\mu$, we get 
	\eqref{e:W_m_vero}.
	We now notice that for every $\lambda$ we have
	$$\ds\left(\frac{\alpha^2+\lambda}{d+2\alpha-2}-\mu\right)-(1-\kappa_{\alpha,\mu})\left(\frac{\mu^2+\lambda}{d+2\mu-2}-\mu\right)
	=\frac{\kappa_{\alpha,\mu}}{d+2\alpha-2}(\lambda_\alpha-\lambda),$$
	which shows \eqref{e:W_m}. 
\end{proof}

\subsection{Energy of homogeneous minimizers} In this subsection we prove a lemma about the energy of homogeneous minimizers which will be useful in their classification.

\begin{lemma}\label{l:mumut}
	Let $\mu\ge 0$ and $t\in\R$. If the trace $c\in H^1(B_1)$ is such that the $(\mu+t)$-homogeneous extension $r^{\mu+t} c(\theta)$ is a solution of the thin-obstacle problem, then 
	\begin{equation}\label{e:mumut}
	\GG_{\mu}(r^{\mu+t}c)=t\|c\|_{L^2(\partial B_1)}^2\qquad\text{and}\qquad \GG_{\mu}(r^{\mu}c)=\left(1+\frac{t}{2\mu+d-2}\right)\GG_{\mu}(r^{\mu+t}c).
	\end{equation}
\end{lemma}
\begin{proof}
Since the Weiss energy vanishes for minimizers with the corresponding homogeneity, $\GG_{\mu+t}(r^{\mu+t} c(\theta))=0$, we get that 
	$$\|\nabla_\theta c\|_{L^2(\partial B_1)}^2=\lambda(\mu+t)\|c\|_{L^2(\partial B_1)}^2.$$
	Hence, we have $\ds \GG_{\mu}(r^{\mu+t}c)=\GG_{\mu+t}(r^{\mu+t}c)+t\|c\|_{L^2(\partial B_1)}^2=t\|c\|_{L^2(\partial B_1)}^2$ and by Lemma \ref{l:thin_fourier} \eqref{e:W_m_vero}
	\begin{align*}
	\GG_{\mu}(r^{\mu}c)&=\frac{1}{2\mu+d-2}(\|\nabla_\theta c\|_{L^2(\partial B_1)}^2-\lambda(\mu)\| c\|_{L^2(\partial B_1)}^2)=\frac{\lambda(\mu+t)-\lambda(\mu)}{2\mu+d-2}\| c\|_{L^2(\partial B_1)}^2\\
	&=\left(1+\frac{t}{2\mu+d-2}\right)t\| c\|_{L^2(\partial B_1)}^2=\left(1+\frac{t}{2\mu+d-2}\right)\GG_{\mu}(r^{\mu+t}c).
	\end{align*}
\end{proof}

\section{Epiperimetric inequality for the regular points: Proof of Theorem \ref{t:epi:3/2}.}\label{sec:epi-reg}

In this section, after some preliminary considerations about $\sfrac32$-homogeneous minimizers of $\mathcal E$, we prove the epiperimetric inequality at regular points Theorem \ref{t:epi:3/2}.

\subsection{Global minimizers of frequency $\sfrac32$} In \cite{AtCaSa} Athanasopoulos, Caffarelli and Salsa notice that there are no point of frequency smaller than $\sfrac32$. On the other hand, one can easily construct global $\sfrac32$-homogeneous solution for which the point $0$ is on the free boundary. In dimension two, one such a solution expressed in polar coordinates is $\ds h_{\sfrac 32}(r,\theta)=r^{\sfrac32}\cos\left({3\theta}/2\right)$, for $r>0$ and $\theta\in(-\pi,\pi)$. In $\R^d$, it is sufficient to consider the two-dimensional solution $h_{\sfrac32}$ extended invariantly in the remaining $d-2$ coordinates. More generally, for a given direction $e\in\partial B_1\cap \{x_d=0\}$ we consider the function $h_e$ in \eqref{eqn:he}, which is a $\sfrac32$-homogeneous global solution of the thin obstacle problem. With a slight abuse of the notation, in polar coordinates, we will sometimes write $h_e(r,\theta)=r^{3/2}h_e(\theta)$. We notice that $h_e$ has the following properties: 
\begin{enumerate}[(i)]
	\item The $L^2(\partial B_1)$-projection of $h_e(\theta)$ on the space of linear functions is non-zero and is given by $c\,x\cdot e$, with $c>0$. Notice that the space of linear functions coincides with the eigenspace of the spherical laplacian corresponding to the eigenvalue $\lambda_2=\dots=\lambda_d=d-1$. Thus, $h_e$ has a non-zero $(d-1)$-mode on the sphere.
         \item $h_e$ is harmonic on $B_1\setminus(\{x_d=0\}\cap \{x\cdot e>0\})$. Thus, an integration by parts gives that, for every $\psi\in H^1(B_1)$ such that $\psi =0$ on $\{x_d=0\}\cap \{x\cdot e<0\}$ we have 
	$$\int_{B_1}\nabla h_e\cdot\nabla\psi\,dx-\frac32\int_{\partial B_1} h_e\psi\,d\HH^{d-1} =0.$$
	In particular, $\GG_{3/2}(h_e)=0$. 
	\item The derivative $\frac{\partial h_e}{\partial x_d}$ has a jump across the set $\{x_d=0\}\cap\{x\cdot e<0\}$. 
	The distributional laplacian of $h_e$ on $B_1$, applied to the test function $\psi\in H^1(B_1)$, is given by 
	$$\int_{B_1}\psi\Delta h_e\,dx=2\int_{B_1'\cap \{x\cdot e<0\}}\left|\frac{\partial h_e}{\partial x_d}\right|\psi\,d\HH^{d-1}=2\int_{B_1'}\ \psi( x',0)\frac{3}{\sqrt 2}( x'\cdot e)_-^{1/2}\,d x'\,.$$
\end{enumerate}

\subsection{Proof of Theorem \ref{t:epi:3/2}}
Since $c$ is even with respect to the plane $\{x_d=0\}$, the projection of $c$ on the eigenspace of linear functions $E(\lambda_2)\subset H^1(B_1)$ is of the form $c_1\, x\cdot e$ for some constant $c_1>0$ and $e\in\partial B_1\cap\{x_d=0\}$. 
Let $C>0$ be such that the $L^2(\partial B_1)$-projections of $C\,h_e$ and $c$ on the eigenspace of linear functions $E(\lambda_2)$ are the same. 

Consider the function $u_0:B_1\to\R$ given by $u_0(x):=|x_d|^{3/2}$. Since $u_0(\theta)$ is even, it is orthogonal to the eigenspace $E(\lambda_2)$. Let the constant $c_0\in\R$ be such that the projections of $c-C h_e$ and $c_0u_0$ on the eigenspace $E(\lambda_1)$ are the same. 

We can now deduce that  $c:\partial B_1\to \R$ can be decomposed in a unique way as $C h_e+c_0u_0$, which has the same low modes of $c$, and of $\phi$, which contains only higher modes on $\partial  B_1$
$$c=C h_e+c_0u_0+\phi, \qquad \phi(\theta)=\sum_{\{j\,:\,\lambda_j>2d=\lambda_2\}}c_j\phi_j(\theta).$$  
The competitor $v:B_1\to\R$ is then given by 
\begin{equation}\label{e:competitor:3/2}
v(r,\theta)=C r^{3/2} h_e(\theta)+c_0 r^{3/2}u_0(\theta)+r^2\phi(\theta),
\end{equation}

\noindent We notice that since $c>0$ on the equator $\{x_d=0\}\cap \partial B_1$ and since $C>0$, assures that  $v(r,\theta) \geq r^2 c(\theta)$ is non-negative on the $(d-1)$-dimensional ball $B_1':=\{x_d=0\}\cap B_1$. 
We compute the energies of $r^{3/2}c$ and $v$.
We first show that, for any $\alpha$-homogeneous function $\psi(r,\theta)=r^\alpha\phi(\theta)$ with $\phi \in H^1(\partial B_1)$, we have  
\begin{equation}\label{e:step1}
\G(C h_e+c_0 u_0+\psi)=-\frac{3c_0^2}4\int_{B_1}|x_d|\,dx+\G(\psi)+\frac{1}{d+\alpha-\frac12}\beta(\phi),
\end{equation}
where 
$$\beta(\phi):=-\frac32c_0\int_{\partial B_1}\frac{\phi(\theta)}{\sqrt{|\theta_d|}}\,d\HH^{d-1}(\theta)+\frac{12}{\sqrt 2}C\int_{\partial B_1'}\phi(\theta')(\theta'\cdot e)_-^{1/2}\,d\HH^{d-2}(\theta').$$
Indeed, expanding $\GG_{3/2}$ and integrating by parts we get 
\begin{align}
\G(C h_e+c_0 u_0+\psi)&=C^2\G(h_e)+\G(c_0u_0+\psi)\nonumber\\
&\qquad+2C\left(\int_{B_1}\nabla h_e\cdot\nabla (c u_0+\psi)-\frac32\int_{\partial B_1} h_e (c_0 u_0+\psi)\right)\nonumber\\
&=\G(c_0u_0+\psi)-2C\int_{B_1}\ \psi \Delta h_e\,dx,\nonumber
\end{align}

\begin{align}
\G(c_0u_0+\psi)&=c_0^2\G(u_0)+\G(\psi)+2c_0\left(\int_{B_1}\nabla u_0\cdot\nabla\psi-\frac32\int_{\partial B_1} u_0\psi\right)\nonumber\\
&=c_0^2\G(u_0)+\G(\psi)-2c_0\int_{B_1}\ \psi \Delta u_0\,dx.\label{e:stima}
\end{align}
An integration by parts and the fact that $\Delta u_0(x)=\frac34|x_d|^{-1/2}$ give that 
\begin{equation}\label{e:Delta_u_0}
\GG_{\sfrac32}(u_0)=-\int_{B_1}u_0\Delta u_0\,dx=-\frac34\int_{B_1}|x_d|\,dx<0\,.
\end{equation}
The $\alpha$-homogeneity of $\psi$ and the precise expressions of $\Delta u_0$ and $\Delta h_e$ give that 
\begin{equation}\label{e:Delta_u_0_psi}
\int_{B_1}\ \psi \Delta u_0\,dx=\int_{B_1}\psi \frac34|x_d|^{-1/2}\,dx
=\frac{1}{d+\alpha-\frac12}\int_{\partial B_1}\phi(\theta)\frac34|\theta_d|^{-1/2}\,d\HH^{d-1}(\theta),
\end{equation}
\begin{equation}\label{e:Delta_h_e_psi}
\int_{B_1}\ \psi \Delta h_e\,dx=-2\int_{B_1'}\ \psi(x',0)\frac{3}{\sqrt 2}(x'\cdot e)_-^{1/2}\,dx'=-\frac{1}{d+\alpha-\frac12}\frac{6}{\sqrt 2}\int_{\partial B_1'}\phi(\theta')(\theta'\cdot e)_-^{1/2}\,d\HH^{d-2}(\theta')\ .
\end{equation}
Finally, by \eqref{e:stima}, \eqref{e:Delta_u_0}, \eqref{e:Delta_u_0_psi} and \eqref{e:Delta_h_e_psi} we get \eqref{e:step1}. Applying  \eqref{e:step1} to $\alpha=\sfrac32$ and $\alpha=2$ we get 
 \begin{equation}\label{e:step2}
\begin{array}{ll}
\ds\G(v)-\left(1-\frac{1}{2d+3}\right)\G(z)&\ds\le -\frac{3\,c_0^2}{4(2d+3)}\int_{B_1}|x_d|\,dx\\
&\qquad \ds  +\,\G(r^2\phi(\theta))-\left(1-\frac{1}{2d+3}\right)\G(r^{\sfrac32}\phi(\theta))\\
&\ds\le -\frac{3\,c_0^2}{4(2d+3)}\int_{B_1}|x_d|\,dx,
\end{array}
\end{equation}
where the last inequality is due to Lemma \ref{l:thin_fourier} with $\mu=3/2$, which concludes the proof. \qed
\begin{oss}\label{oss:3/2}
In this remark we are interested in the equality case of the epiperimetric inequality \eqref{e:epi:3/2}. Indeed, if there was an equality in  \eqref{e:epi:3/2}, then by \eqref{e:step2} we should have that $c_0=0$ and also 
$$\G(r^2\phi(\theta))-\left(1-\frac{1}{2d+3}\right)\G(r^{\sfrac32}\phi(\theta))=0.$$
By Lemma \ref{l:thin_fourier}, we get that $\phi$ is an eigenfunction on the sphere $\partial B_1$ corresponding to the eigenvalue $\lambda(2)=2d$, that is the restriction of a $2$-homogeneous harmonic polynomial. Moreover, since the trace is $c$ is non-negative on $\partial B_1'$  and $h_e=0$ on $B_1'\cap\{x\cdot e<0\}$ we get that $\phi\ge 0$ on $B_1'\cap\{x\cdot e<0\}$ and by the fact that $\phi$ is even, we get $\phi\ge 0$ on $B_1'$. 
\end{oss}


\section{Logarithmic epiperimetric inequality for $2m$-singular points: Proof of Theorem \ref{t:epi:2m}}\label{sec:epi-sing}



If $\GG_{2m}(z)\le 0$, the conclusion is trivial, taking $h\equiv z$. Thus in the proof we assume $\GG_{2m}(z)>0$.

We decompose the trace $c:\partial B_1\to\R$ in Fourier series as 
$$
c(\theta) = \sum_{j=1}^\infty c_j \phi_j(\theta),$$
where by $\phi_j$ we denote the eigenfunctions of the Laplacian on the sphere, by $\lambda_j$ the corresponding eigenvalues and by $\alpha_j$ the corresponding homogeneities (see Subsection \ref{sub:fourier}), and we set 
\begin{equation}\label{e:epi:2m:Pphi}
P(\theta)\quad:\,=\sum_{\{j\,:\,\alpha_j\le 2m\}} c_j\,\phi_j(\theta)\qquad\text{and}\qquad \phi(\theta)\quad:\,=\sum_{\{j\,:\,\alpha_j> 2m\}} c_j\,\phi_j(\theta)\,.
\end{equation}
Let 
$$M:=-\min\big\{\min\{P(\theta),0\}\ :\ \theta\in\partial B_1,\ \theta_d=0\big\},$$
and let $h_{2m}$ be an eigenfunction, corresponding to the homogeneity $2m$, such that $h_{2m}\equiv1$ on the hyperplane $\{x_d=0\}\cap\partial B_1$. 
\begin{oss}[Construction of $h_{2m}$]\label{rmk:h2m}
In order to construct such an eigenfuction we first notice that the eigenspace corresponding to the homogeneity $2m$ consists of the restrictions to the sphere of $2m$-homogeneous harmonic polynomials in $\R^d$. Thus it is sufficient to construct a $2m$-homogeneous harmonic polynomial whose restriction to the space $\{x_d=0\}$ is precisely $\big(x_1^2+\dots+x_{d-1}^2\big)^m$. We define 
$$h_{2m}(x_1,\dots,x_d):=\sum_{n=0}^m C_n x_d^{2n}(x_1^2+\dots+x_{d-1}^2)^{m-n},$$
where $C_0=1$ and, for every $n\ge 1$, $C_n$ is given by the formula
$$C_{n}:=-\frac{2(m-n+1)(d-1+2m-2n)}{2n(2n-1)}C_{n-1}.$$
It is immediate to check that $C_n$ is explicitely given by
$$C_n=\frac{(-2)^n\,m!}{(2n)!\,(m-n)!}\prod_{j=1}^n(d-1+2m-2j),$$
which concludes the construction of $h_{2m}$. 
\end{oss}
The $2m$-homogeneous extension $z$ of $c$ can be written as
$$z(r,\theta)=r^{2m}P(\theta)+M\, r^{2m} h_{2m}(\theta)-M\, r^{2m}h_{2m}(\theta)+r^{2m}\phi(\theta).$$
Our competitor $h$ is given by 
\begin{equation}\label{e:competitor:2m}
h(r,\theta)=r^{2m}P(\theta)+M\, r^{2m} h_{2m}(\theta)-M\,r^{\alpha}h_{2m}(\theta)+r^{\alpha}\phi(\theta).
\end{equation}
for some $\alpha>2m$ to be chosen later. 
Notice that $h$ is non-negative on the set $\{x_d=0\}\cap\partial B_1$.

We will choose the homogeneity $2m<\alpha\le 2m+1$ such that
\begin{equation}\label{e:choice_of_alpha}
\kappa_{\alpha,2m}:=\frac{\alpha-2m}{\alpha+2m+d-2}=\ds\eps \|\nabla_\theta \phi\|_{L^2(\partial B_1)}^{2\gamma}.
\end{equation}
Subsequently we will choose $\eps$ to be small enough, but yet depending only on the dimension.  We now prove the epiperimetric inequality \eqref{e:epi:2m}. We proceed in three steps. 

\medskip 

\noindent {\emph{Step 1.}} 
There are explicit (given in \eqref{e:2m:energy:constants}) constants $C_1$ and $C_2$, depending only on $d$ and $m$, such that for every $2m<\alpha\le 2m+\frac12$ the following inequality does hold:  
\begin{align}
\GG_{2m}(h)-(1-\kappa_{\alpha,2m})\GG_{2m}(z)\le C_1\kappa_{\alpha,2m}^2 \,M^2-C_2\kappa_{\alpha,2m}\|\nabla_\theta \phi\|_{L^2(\partial B_1)}^2.
\label{e:2m:energy}
\end{align}
We set for simplicity 
\begin{equation}\label{e:2m:notations}
\begin{array}{rcl}
\ds\psi(r,\theta)&:=&\ds\sum_{\{j,\ \alpha_j<2m\}} c_j r^{2m}\phi_j(\theta),\\
\\
\ds H_{2m}(r,\theta)&:=&\ds M\,r^{2m}h_{2m}(\theta)+\sum_{\{j,\ \alpha_j=2m\}} c_j r^{2m}\phi_j(\theta),\\
\\
\ds\varphi(r,\theta)&:=&\ds-M\,r^{2m}h_{2m}(\theta)+\sum_{\{j,\ \alpha_j>2m\}} c_j r^{2m}\phi_j(\theta),\\
\\
\ds\tilde\varphi(r,\theta)&:=&\ds-M\,r^{\alpha}h_{2m}(\theta)+\sum_{\{j,\ \alpha_j>2m\}} c_j r^{\alpha}\phi_j(\theta). 
\end{array}
\end{equation}
Thus, $h$ and $z$ are given by 
$$z=\psi+H_{2m}+\varphi\qquad\text{and}\qquad h=\psi+H_{2m}+\tilde\varphi.$$
We first notice that the harmonicity and $2m$-homogeneity of $H_{2m}$ imply
$$\GG_{2m}(z)=\GG_{2m}(\psi+\varphi)\qquad\text{and}\qquad \GG_{2m}(h)=\GG_{2m}(\psi+\tilde\varphi).$$
Moreover, by definition $\psi$ is orthogonal in $L^2(B_1)$ and $H^1(B_1)$ to both $\varphi$ and $\tilde\varphi$. Thus, we get 
$$\GG_{2m}(z)=\GG_{2m}(\psi)+\GG_{2m}(\varphi)\qquad\text{and}\qquad \GG_{2m}(h)=\GG_{2m}(\psi)+\GG_{2m}(\tilde\varphi).$$
We now notice that, since $\psi$ contains only lower frequencies, we have $\GG_{2m}(\psi)<0$. Thus, 
\begin{align*}
\GG_{2m}(h)-(1-\kappa_{\alpha,2m})\GG_{2m}(z)&=\kappa_{\alpha,2m}\GG(\psi)+\GG_{2m}(\tilde\varphi)-(1-\kappa_{\alpha,2m})\GG_{2m}(\varphi)\\
&\le \GG_{2m}(\tilde\varphi)-(1-\kappa_{\alpha,2m})\GG_{2m}(\varphi)
\end{align*}
By  Lemma \ref{l:thin_fourier} we have that
\begin{align}
\GG_{2m}(\tilde\varphi)-(1-\kappa_{\alpha,2m})\GG_{2m}(\varphi)&=  M^2\|h_{2m}\|_{L^2(\partial B_1)}^2\frac{\kappa_{\alpha,2m}}{d+2\alpha-2}(-\lambda(2m)+\lambda(\alpha))\nonumber\\
&\qquad \qquad +\frac{\kappa_{\alpha,2m}}{d+2\alpha-2}\sum_{\{j\,,\,\alpha_j>2m\}}^\infty (-\lambda_j+\lambda(\alpha)) c_j^2\nonumber\\
&=  M^2\|h_{2m}\|_{L^2(\partial B_1)}^2\kappa_{\alpha,2m}^2\frac{(2m+\alpha+d-2)^2}{d+2\alpha-2}\nonumber\\
&\qquad \qquad+\frac{\kappa_{\alpha,2m}}{d+2\alpha-2}\sum_{\{j\,,\,\alpha_j>2m\}}^\infty (-\lambda_j+\lambda(\alpha)) c_j^2.\label{e:2m:energy:mainest}
\end{align}
If we consider the further restriction $2m<\alpha\le 2m+\sfrac12$, then there is a constant $C_2>0$, depending on $d$ and $m$, such that 
\begin{align}
\sum  (\lambda_j-\lambda(\alpha))c_j^2&=\sum \lambda_j c_j^2-\lambda(\alpha)\sum c_j^2\ge \sum \lambda_j c_j^2-\frac{\lambda_\alpha}{\lambda(2m+1)}\sum \lambda_jc_j^2\nonumber\\
&\ge \sum \lambda_j c_j^2-\frac{\lambda(2m+\frac12)}{\lambda(2m+1)}\sum \lambda_jc_j^2\ge  C_2 \sum \lambda_j c_j^2=C_2 \|\nabla_\theta \phi\|_{L^2(\partial B_1)}^{2},\label{e:salviamo_le_orche}
\end{align}
where all the sums are over $\{j\,,\,\alpha_j>2m\}$. Combinig \eqref{e:salviamo_le_orche} with \eqref{e:2m:energy:mainest} we get \eqref{e:2m:energy} with 
\begin{equation}\label{e:2m:energy:constants}
C_1=(4m+d)\|h_{2m}\|_{L^2(\partial B_1)}^2\qquad\text{and}\qquad C_2=\frac{\lambda(2m+1)-\lambda(2m+\frac12)}{\lambda(2m+1)}.\end{equation}
We conclude this Step 1 of the proof by noticing that we implicitly used the bounds on $c$ in \eqref{e:orchecattive}.
Indeed, in order to have the restriction $\alpha< 2m+\frac12$ we need an explicit bound, in terms of $d$ and $m$, on the norm $\|\nabla_\theta\phi\|_{L^2(\partial B_1)}$. Repeating the estimate \eqref{e:salviamo_le_orche} with $\alpha=2m$ we get that there is a constant $C_{d,m}$, depending on $m$ and $d$, such that
\begin{equation}\label{e:orche_volanti}
\|\nabla_\theta \phi\|_{L^2(\partial B_1)}^{2}\le C_{d,m}\!\!\!\!\!\!\!\sum_{\{j\,,\,\alpha_j>2m\}}  (\lambda_j-\lambda(2m))c_j^2= C_{d,m}\GG_{2m}(r^{2m}\phi(\theta)),\end{equation}
where the last inequality is due to Lemma \ref{l:thin_fourier}, equation \eqref{e:W_m_vero}. Using that $c=P+\phi$ and the orthogonality of $P$ and $\phi$ on the sphere, we get that 
\begin{align*}
\GG_{2m}(r^{2m}\phi(\theta))&\le \GG_{2m}(z)-\GG_{2m}(r^{2m}P(\theta))\le \GG_{2m}(z)+\frac{\lambda(2m)}{2m+d-2}\|P\|_{L^2(\partial B_1)}^2\\
&\le \GG_{2m}(z)+2m\|c\|_{L^2(\partial B_1)}^2\le 1+2m,
\end{align*}
which together with \eqref{e:orche_volanti} proves that there is a constant $C_{d,m}$ such that 
$$\|\nabla_\theta \phi\|_{L^2(\partial B_1)}^{2}\le C_{d,m}.$$
Thus, choosing $\eps\le \frac{1}{2C_{d,m}}$, the condition $\alpha-2m\le \sfrac12$ is satisfied for every trace $c$ for which \eqref{e:orchecattive} does hold.

\noindent{\it Step 2.} There is a constant $C_3>0$, depending on $d$ and $m$, such that 
\begin{equation}\label{e:2m:gamma}
M^2\le C_3\|\nabla_\theta \phi\|_{L^2(\partial B_1)}^{2(1-\gamma)}.
\end{equation}
We start by noticing that there is a constant $L_m$, depending only on $d$ and $m$, such that the eigenfunctions corresponding to the low frequencies are globally $L_m$-Lipschitz continuous, that is
$$\|\nabla_\theta \phi_j\|_{L^\infty(\partial B_1)}\le L_m\,,\quad\text{for every}\quad j\in\N\quad\text{such that}\quad \alpha_j\le 2m.$$
Now, since by hypothesis the trace $c(\theta)$ is such that $\|P\|_{L^2(\partial B^1)}^2\le \|c\|_{L^2(\partial B^1)}^2\le\Theta$, we have that all the constants $c_j$ in the Fourier expansion of $P$ are bounded by $\sqrt\Theta$. Thus, the function $P:\partial B_1\to\R$ is $L$-Lipschitz continuous for some $L>0$, depending on $d$, $m$ and $\Theta$. Denoting by $P_-$ the negative part of $P$, $P_-(\theta)=\min\{P(\theta),0\}$, we get that 
\begin{equation}\label{e:2m:gamma:lipest}
\int_{\mathbb{S}^{d-2}} P_-^2\,d\HH^{d-2}\ge C_d M^2 \left(\frac{M}{L}\right)^{d-2}=\frac{C_d}{L^{d-2}}M^{d},
\end{equation}
for some dimensional constant $C_d$. On the other hand, since $P+\phi$ is non-negative on $\mathbb{S}^{d-2}=\{x_d=0\}\cap\partial B_1$ we get that 
\begin{equation}\label{e:2m:gamma:positivity}
\int_{\mathbb{S}^{d-2}} \phi^2\,d\HH^{d-2}\ge \int_{\mathbb{S}^{d-2}} P_-^2\,d\HH^{d-2}.
\end{equation}
Now, by the trace inequality on the sphere $\partial B_1$, there is a dimensional constant $C_d$ such that 
\begin{equation}\label{e:2m:gamma:trace}
\begin{array}{ll}
\ds\int_{\mathbb{S}^{d-2}} \phi^2\,d\HH^{d-2}&\ds\le C_d \left(\int_{\mathbb{S}^{d-1}} |\nabla_\theta\phi|^2\,d\HH^{d-1}+\int_{\mathbb{S}^{d-1}} \phi^2\,d\HH^{d-1}\right)\\
&\ds\le C_d \left(1+\frac1{\lambda(2m)}\right)\int_{\mathbb{S}^{d-1}} |\nabla_\theta\phi|^2\,d\HH^{d-1},
\end{array}
\end{equation}
where the last inequality is due to the fact that in the Fourier expansion of $\phi$ there are only frequencies $\lambda_j>\lambda(2m)$. Combining \eqref{e:2m:gamma:lipest}, \eqref{e:2m:gamma:positivity} and \eqref{e:2m:gamma:trace}, we get \eqref{e:2m:gamma}. 

Notice that in this step we used the non-negativity of the trace $c$ (in the inequality \eqref{e:2m:gamma:positivity}) and also the condition that $c$ is bounded in $L^2(\partial B_1)$ (when we give the Lipschitz bound on $P$). More precisely, the constant $C_3$ depends on the norm $\|P\|_{L^2(\partial B_1)}$, which in turn is bounded by $\Theta$.

\medskip

\noindent{\it Step 3. Conclusion of the proof of Theorem \ref{t:epi:2m}.} 
Combining the inequalities \eqref{e:2m:energy} and \eqref{e:2m:gamma} we get
\begin{align*}
\GG_{2m}(h)-(1-\kappa_{\alpha,2m})\GG_{2m}(z)&\le C_1\kappa_{\alpha,2m}^2 M^2-C_2\kappa_{\alpha,2m}\|\nabla_\theta \phi\|_{L^2(\partial B_1)}^2\\
&\le \kappa_{\alpha,2m}^2 C_1C_3\|\nabla_\theta \phi\|_{L^2(\partial B_1)}^{2(1-\gamma)}-C_2\kappa_{\alpha,2m}\|\nabla_\theta \phi\|_{L^2(\partial B_1)}^2.
\end{align*}
By the definition of $\kappa_{\alpha,2m}$ we have  
\begin{align}
\GG_{2m}(h)-(1-\kappa_{\alpha,2m})\GG_{2m}(z)&\le \eps^2\|\nabla_\theta \phi\|_{L^2(\partial B_1)}^{4\gamma} C_1C_3\|\nabla_\theta \phi\|_{L^2(\partial B_1)}^{2(1-\gamma)}-C_2\eps \|\nabla_\theta \phi\|_{L^2(\partial B_1)}^{2\gamma} \|\nabla_\theta \phi\|_{L^2(\partial B_1)}^2\notag\\
&= \eps \left(\eps C_1C_3 -C_2\right) \|\nabla_\theta \phi\|_{L^2(\partial B_1)}^{2+2\gamma}\le  0,\label{e:epi:2m:stronger}
\end{align}
where, in order to have the last inequality, we choose  $\eps$ such that
$$0<\eps\le \frac{C_2}{C_1C_3}.$$
We now notice that , by Lemma \ref{l:thin_fourier}, we have 
\begin{align}
\GG_{2m}(z)&=\frac{1}{4m+d-2}\sum_{j=1}^\infty(\lambda_j-\lambda(2m))c_j^2\nonumber\\
&\le \frac{1}{4m+d-2}\sum_{\{j\,,\,\alpha_j>2m\}}^\infty(\lambda_j-\lambda(2m))c_j^2\le \sum_{\{j\,,\,\alpha_j>2m\}}^\infty\lambda_jc_j^2= \|\nabla_\theta\phi\|_{L^2(\partial B_1)}^2.\label{e:2m:final}
\end{align}
Thus, we get
\begin{align*}
\GG_{2m}(h)&\le \left(1-\kappa_{\alpha,2m}\right)\GG_{2m}(z)= \left(1-\eps\, \|\nabla_\theta \phi\|_{L^2(\partial B_1)}^{2\gamma}\right)\GG_{2m}(z) \le \big(1-\eps\,\GG_{2m}^\gamma (z)\big)\GG_{2m}(z),
\end{align*}
which is precisely \eqref{e:epi:2m}. Finally, we notice that in this last step of the proof we didn't use any specific condition on the trace $c$. 
\qed

We conclude this section with the following Remark, which will be useful for the characterization of the possible blow-up limits.
\begin{oss}\label{oss:epi:2m:improved}
In the hypotheses of Theorem \ref{t:epi:2m}, we have the following, slightly stronger version of the logarithmic epiperimetric inequality:  
\begin{equation}\label{e:epi:2m:improved}
\GG_{2m}(h)\le \GG_{2m}(z) \big(1-\eps \,|\GG_{2m}(z)|^{\gamma}\big)-\frac{C_2\eps}{2}\|\nabla_\theta\phi\|_{L^2(\partial B_1)}^{2+2\gamma},
\end{equation}
for which it is sufficient to choose $0<\eps<\frac{C_2}{2C_1C_3}$ 
in \eqref{e:epi:2m:stronger}, $\phi$ being the function containing the higher modes of the trace $c$ on the sphere $\partial B_1$ (see \eqref{e:epi:2m:Pphi}).
\end{oss}

\hfill

\section{Epiperimetric inequality for the points of frequency $2m-\sfrac12$ in dimension two. Proof of Theorem \ref{t:epi:2m-1/2}}\label{sec:epi-2d}
We prove the theorem in several steps. \\
{\it Step 1. Sectorial decomposition of $h_{2m-\sfrac 12}$.} We notice that the function $h_{2m-\sfrac 12}$ has $4m-1$ half-lines from the origin along which it vanishes. These lines correspond to the angles
$$s_i :=\frac{2 i}{4m-1} \pi, \qquad \mbox{for }i=1,...,4m-1,$$
 and they individuate $4m-1$ circular sectors in $B_1$ corresponding to the nodal domains of $h_{2m-\sfrac 12}$. We consider the following $2m$ sets, which are invariant under the transformation $\theta \to -\theta$
 $$S_j=\big\{ (r,\theta): r\in [0,1], \; \theta \in\, ] s_{j-1},  s_j[ \,\cup\, ]2\pi-s_{j},  2\pi-s_{j-1}[\big\},$$
 where $j=1,\dots,2m$.
 Notice that $S_1,\dots,S_{2m-1}$ are unions of two sectors of angle $\frac{2\pi}{4m-1}$, while $S_{2m}$ is the sector $\big\{ (r,\theta): r\in [0,1], \; \theta \in \,] s_{2m-1},  s_{2m}[\big\}$. 
We define the restrictions of $h_{2m-\sfrac 12}$ to these sectors for $j=1,...,2m$
$$f_j (r,\theta):= \ind_{S_j}(r,\theta) h_{2m-\sfrac 12}(r,\theta) = \Big( \ind_{] s_{j-1},  s_j[} (\theta) + \ind_{]2\pi-s_{j},  2\pi-s_{j-1}[ }(\theta) \Big)
h_{2m-\sfrac 12}(r,\theta)
. $$
We notice that, since $h_{2m-\sfrac 12}$ vanishes on $B_1\cap \partial S_j$, the fuctions $f_j$ are in $H^1(B_1)$. Moreover, they are $(2m-\sfrac12)$-homogeneous even functions, namely $f_j(r,\theta)=r^{2m-1/2}f_j(\theta)$ and $f_j(r,\theta) = f_j(r,-\theta)$. We claim that for any $b_1,..., b_{2m} \in \R$
\begin{equation}
\label{eqn:fette-en-0}
\GG_{2m-1/2} \Big(\sum_{i=1}^{2m}b_{i}f_{i}\Big) =0.
\end{equation}
Indeed, since the energy $\GG_{2m-1/2}$ is quadratic in its argument and for every $i \neq j$ the supports of $f_i$ and $f_j$ have negligible intersection, the energy of the linear combination is given by 
$$
\GG_{2m-1/2} \Big(\sum_{i=1}^{2m}b_{i}f_{i}\Big) = \sum_{i=1}^{2m}b_i^2\GG_{2m-1/2} (f_i). 
$$
Moreover the functions $f_i$ are harmonic in each $S_j$ and vanish on the rays delimiting their support, that is on $\partial S_j\cap B_1$. Thus $\GG_{2m-1/2} (f_i) = 0$ for every $i=1,..., 2m$, so that the previous inequality implies \eqref{eqn:fette-en-0}.\\

\noindent {\it Step 2. Decomposition of the datum $c$.}
We claim that we can write $c$ in a unique way as 
$$c(\theta)= \sum_{i=1}^{2m} a_i f_{i}(\theta) + \tilde c(\theta) \qquad \mbox{on}\quad \partial B_1,$$
where
\begin{itemize}
\item $a_1, ..., a_{2m} \in \R$ and $a_{2m} >0$
\item $ \tilde c \in H^1(\partial B_1)$ is even and it is orthogonal in $L^2(\partial B_1)$ to $1, \cos(\theta), ... , \cos((2m-1)\theta)$.
\end{itemize}
To prove this claim, we call $L$ the span of $1$, $cos(\theta)$,..., $\cos((2m-1)\theta)$, which is a linear subspace of $L^2(\partial B_1)$ of dimension $2m$. We set $P_L(c)$ to be the projection of $c$ onto $L$. To show the existence of $a_1,..., a_{2m} \in \R$, it is enough to prove that the $2m$ functions $P_L(f_1),..., P_L(f_{2m})$  are linearly independent, so that their span gives the whole $L$. Hence, we take any linear combination $b_1f_1+ ... +b_{2m}f_{2m}$, such that its projection on $L$ is $0$, aiming to prove that $b_1= ... =b_{2m}=0$. By \eqref{eqn:fette-en-0}, the energy of $b_1f_1+ ... +b_{2m}f_{2m}$ is $0$. On the other hand, since the function $b_1f_1+ ... +b_{2m}f_{2m}$ is assumed to have only modes higher than $2m-\sfrac12$ on $\partial B_1$, its $(2m-\sfrac12)$-homogenous extension has nonnegative energy thanks to \eqref{e:W_m_vero}, and its energy is $0$ if and only if $b_1f_1+ ... +b_{2m}f_{2m} \equiv 0$. Hence, this must be the case. 
Hence we can write in a unique way $P_L(c)$ as a linear combination of $P_L(f_1),..., P_L(f_{2m})$
$$P_L(c) =  \sum_{i=1}^{2m} a_i P_L(f_{i}).$$ 
Since $c$ is assumed to be close to $h_{2m-\sfrac 12}$ by \eqref{hp:vicino}, and since $h_{2m-\sfrac 12}$ is strictly positive on the support of $f_{2m}$, we can assume without loss of generality that  $a_{2m}>0$. Finally, we set $\tilde c = c- \sum_{i=1}^{2m} a_i f_{i} $.\\

\noindent{\it Step 3. Choice of an energy competitor and computation of the energy.}
We let $\alpha> 2m-1/2$ to be chosen later and we define an energy competitor for $c$ as 
$$h(r,\theta):=\sum_{j=1}^{2m} a_j f_j(r,\theta) + r^{\alpha}\tilde c(\theta)  = r^{\frac{4m-1}{2}} (c- \tilde c) + r^{\alpha}\tilde c 
.$$

The energy of $h$ can be written as
\begin{equation*}
\begin{split}
\GG_{2m-\sfrac12}(h)
= &\GG_{2m-\sfrac12} \big( r^{\frac{4m-1}{2}} (c- \tilde c) \big) +\GG_{2m-1/2} \big( r^{\alpha} \tilde c\big) 
\\&+ 2\int_{B_1} \nabla \big(r^{\frac{4m-1}{2}} (c- \tilde c)\big) \cdot \nabla \big( r^{\alpha} \tilde c \big) \, d\HH^2
-(4m-1) \int_{\partial B_1} (c- \tilde c) \tilde c \, d\HH^1.
\end{split}
\end{equation*}
By the deinition of $\tilde c$ and Step 1 we have that the first term in the right-hand side vanishes:
\begin{equation}\label{e:tuturutu}
\begin{split}
\!\!\!\GG_{2m-\sfrac12}(h)= &\GG_{2m-\sfrac12} \big( r^{\alpha} \tilde c\big) 
+ 2  \sum_{j=1}^{2m} a_j \left(\int_{B_1} \nabla f_j \cdot \nabla \big( r^{\alpha} \tilde c \big) \, d\HH^2
-(4m-1)   \int_{\partial B_1} f_j \tilde c \, d\HH^1\right).
\end{split}
\end{equation}
We rewrite the middle term integrating by parts, and using that $\Delta f_j = 0$ on $\{f_j\neq0\}$
\begin{equation*}
\begin{split}
\int_{B_1}& \nabla f_j \cdot \nabla \big( r^{\alpha} \tilde c \big) \, d\HH^2
= \int_{S_j} r^{\alpha} \tilde c(\theta)\,    \Delta f_j  + \int_{\partial S_j} \frac{\partial f_j}{\partial n} \, r^{\alpha} \tilde c
\\&=2 \int_{\partial B_1^+ \cap S_j} \frac{\partial f_j}{\partial r}r^{\alpha} \tilde c(\theta)
 \, d\HH^1+ 2 \int_{\{\theta =s_j\}} \frac{1}{r} \frac{\partial f_j}{\partial \theta} r^{\alpha} \tilde c(\theta) \, d\HH^1 - 2 \int_{\{\theta = s_{j-1}\}} \frac{1}{r} \frac{\partial f_j}{\partial \theta}r^{\alpha} \tilde c(\theta) \, d\HH^1. 
 \end{split}
\end{equation*}
Now since $f_j$ is $(2m-\sfrac12)$-homogeneous we can write $f_j(r,\theta)=r^{2m-1/2}f_j(\theta)$ and we get that
\begin{equation*}
\begin{split}
2 \int_{\partial B_1^+ \cap S_j} \frac{\partial f_j}{\partial r}r^{\alpha} \tilde c(\theta)\, d\HH^1&=2 \int_{\partial B_1^+ \cap S_j}\!\!\! \partial _r(r^{2m-1/2}f_j(\theta))r^{\alpha} \tilde c(\theta)
 \, d\HH^1=(4m-1) \int_{\partial B_1^+ } f_j \tilde c
 \, d\HH^1,\\
  2 \int_{\{\theta =s_j\}} \frac{1}{r} \frac{\partial f_j}{\partial \theta} r^{\alpha} \tilde c(\theta) \, d\HH^1 &- 2 \int_{\{\theta = s_{j-1}\}} \frac{1}{r} \frac{\partial f_j}{\partial \theta}r^{\alpha} \tilde c(\theta) \, d\HH^1\\
  &= 2 \int_0^1 r^{\alpha+\frac{4m-3}{2}} \left(\partial_{\theta} f_j(s_{j})\tilde c(s_{j})-\partial_{\theta} f_j(s_{j-1})\tilde c(s_{j-1})\right) \, dr \\
  &= \frac{2}{\alpha+2m-1/2}\big(\partial_{\theta} f_j(s_{j})\tilde c(s_{j})-\partial_{\theta} f_j(s_{j-1})\tilde c(s_{j-1})\big).
 \end{split}
\end{equation*}
Hence we can rewrite \eqref{e:tuturutu} as
\begin{equation}
\label{eqn:2d-en-alpha}
\GG_{2m-\sfrac12}(h)
= \GG_{2m-\sfrac12} \big( r^{\alpha} \tilde c\big) + \frac{2}{\alpha+2m-1/2}\sum_{j=1}^{2m} a_j\big(\partial_{\theta}f_j(s_{j})\tilde c(s_{j})  -\partial_{\theta}  f_j (s_{j-1})\tilde c(s_{j-1})\big).
\end{equation}
Since the previous two equalities hold also when $\alpha=2m-1/2$, we see that 
\begin{equation}
\label{eqn:2d-en-mu}
\GG_{2m-\sfrac12}(z)
= \GG_{2m-\sfrac12} \big( r^{\frac{4m-1}{2}} \tilde c\big) + \frac{2}{4m-1} \sum_{j=1}^{2m} a_j\big(\partial_{\theta}f_j(s_{j})\tilde c(s_{j})  -\partial_{\theta}  f_j (s_{j-1})\tilde c(s_{j-1})\big).
\end{equation}
\noindent {\it Step 4. Conclusion.} Setting $\kappa_{\alpha,2m-1/2}$ according to \eqref{eps_choice}, a suitable linear combination between the last terms in \eqref{eqn:2d-en-alpha} and \eqref{eqn:2d-en-mu} is $0$, because by the defintion of $\kappa_{\alpha,2m-1/2}$ we have
\begin{equation}
\label{eqn:kappa-helps}
 \frac{2}{\alpha+2m-1/2} - (1-\kappa_{\alpha,2m-1/2})  \frac{2}{4m-1} = 0.
\end{equation}
Putting together \eqref{eqn:2d-en-alpha}, \eqref{eqn:2d-en-mu} and \eqref{eqn:kappa-helps}, we find
$$\GG_{2m-\sfrac12}(h)-(1-\kappa_{\alpha,2m-1/2})\GG_{2m-\sfrac12}(z) = 
\GG_{2m-\sfrac12}\big( r^{\alpha} \tilde c\big)-(1-\kappa_{\alpha,2m-1/2})\GG_{2m-\sfrac12}\big(r^{\frac{4m-1}{2}} \tilde c\big).$$
Thanks to Lemma~\ref{l:thin_fourier}, in particular to \eqref{e:W_m}, we obtain that 
$$\GG_{2m-\sfrac12}\big( r^{\alpha} \tilde c\big)-(1-\kappa_{\alpha,2m-1/2})\GG_{2m-\sfrac12}\big(r^{\frac{4m-1}{2}} \tilde c\big) \leq 0,$$
because by definition $\tilde c$ is orthogonal to $1, \cos(\theta), ... , \cos((2m-1)\theta)$ (which, in dimension $2$, are the only eigenfunctions with corresponding homogeneity less than or equal to $2m-\sfrac 12$).\qed

\section{Admissible frequencies for the thin-obstacle problem}\label{sec:freq}
We first prove an easy version of the epiperimetric inequality useful for negative energies. Then we use this result, together with Lemma \ref{l:mumut} and Theorems \ref{t:epi:3/2} and \ref{t:epi:2m} to conclude the proof of Theorem \ref{t:freq}.

\subsection{Epiperimetric inequality for negative energies} The following proposition gives an epiperimetric inequality for negative energies.

\begin{prop}[Epiperimetric inequality for negative energies]\label{prop:epi:2m-neg}
	Let $d\ge 2$, 
	$c\in H^1(\partial B_1)$ be a function such that 
	its $2m$-homogeneous extension
	$z(r,\theta):=r^{2m}c(\theta)\in \mathcal A$ and $\|c\|_{L^2(\de B_1)}=1$.
	
	Then there exist a constant $\eps=\eps(d,m)>0$ and a function $h\in \mathcal A$ with $h=c$ on $\partial B_1$ and
	\begin{equation}\label{e:epi:2m-neg}
	\GG_{2m}(h)\le  (1+\eps )\GG_{2m}(z).
	\end{equation}
\end{prop}

\begin{proof}For $j\in \N$, let $\phi_j$ be the eigenfunctions of the Laplacian on $\partial B_1$,  $\lambda_j$ and $\alpha_j$ the corresponding eigenvalues and homogeneities (see Subsection \ref{sub:fourier}).
	We decompose $c$ on $\partial B_1$ in Fourier as 
	$$
	c = \sum_{j=1}^\infty c_j \phi_j = \sum_{\{j\,:\,\alpha_j<2m\}} c_j\,\phi_j+\sum_{\{j\,:\,\alpha_j= 2m\}} c_j\,\phi_j + \sum_{\{j\,:\,\alpha_j>2m\}} c_j\,\phi_j =:c_<+ c_=+c_>$$
	We consider the maximum of the negative part of $c_<$ 
	$$M:=-\min\big\{\min\{c_<(\theta),0\}\ :\ \theta\in\partial B_1,\ \theta_d=0\big\}.$$
	Since $Q$ contains only low Fourier frequencies, $M$ is controlled by $\|Q\|_{L^2(\partial B_1)}$, namely, there is a constant $C_1:=C_1(d,m)>0$ such that 
	\begin{equation}\label{e:2m:gamma-neg}
	M^2 \le \Big(\sum_{\{j\,:\,\alpha_j<2m\}} |c_j| \Big)^2 \le C_1 \sum_{\{j\,:\,\alpha_j<2m\}} c_j^2  = C_1\|Q\|^2_{L^2(\partial B_1)}.
	\end{equation}
	Let $\alpha:= \alpha(d,m) \in (2m-1,2m)$ to be chosen later and let
	\begin{equation}\label{e:choice_of_alpha-neg}
	\eps:=\frac{2m-\alpha}{\alpha+2m+d-2}>0.
	\end{equation}
	Let $h_{2m}$ be the eigenfunction built in Remark~\ref{rmk:h2m}, corresponding to the homogeneity $2m$, such that $h_{2m}\equiv1$ on the hyperplane $\{x_d=0\}\cap\partial B_1$. 
	We set for simplicity 
	\begin{equation*}
	h_{<,\mu}(r,\theta):=(c_<(\theta) +M h_{2m}(\theta))r^{\mu} \qquad \mbox{for } \mu = 2m, \alpha,
	\end{equation*}
	\begin{equation*}
	\ds h_=(r,\theta):=\ds (c_= (\theta)- Mh_{2m}(\theta)) r^{2m},\qquad h_>(r,\theta):=c_> r^{2m}\phi_j(\theta).
	\end{equation*}
	We notice that $z$ can be written as a sum of these objects and we introduce the energy competitor $h$, obtaining by extending the lower modes of $c$ with homogeneity $\alpha$ and leaving the rest unchanged
	\begin{equation}\label{e:competitore_senza_nome}
	z=h_{<,2m}+h_{=}+h_{>} \qquad\text{and}\qquad h=h_{<,\alpha}+h_{=}+h_{>}.
	\end{equation}
	Since $0 \leq h_{<,\alpha} \geq h_{<,2m}$ on $B_1'$, we have that $h\geq z\geq 0$ in $B_1'$; moreover, $h=z$ on $\partial B_1$. 
	
	Next, we compute the energy of $z$ and $h$. 
	Since $h_{=}$ is harmonic and $2m$-homogenous, and since  $h_>$ is orthogonal in $L^2(B_1)$ and $H^1(B_1)$ to $h_{<,\mu}$, for $\mu = 2m, \alpha$ we have
	$$\GG_{2m}(h_{<,\mu}+h_{=}+h_{>} )=\GG_{2m}(h_{<,\mu}+h_{=}+h_{>} )
	=\GG_{2m}(h_{<,\mu})+\GG_{2m}(h_{>} ).
	$$
	Thus, we rewrite the quantity in \eqref{e:epi:2m-neg} and we observe that $\GG_{2m}(h_{>})\geq 0$ by  Lemma \ref{l:thin_fourier} 
	\begin{align}
	\GG_{2m}(h)-(1+\eps)\GG_{2m}(z)&=\GG_{2m}(h_{<,\alpha})-(1+\eps)\GG_{2m}(h_{<,2m})-\eps\GG_{2m}(h_{>})\nonumber\\
	&\le \GG_{2m}(h_{<,\alpha})-(1+\eps)\GG_{2m}(h_{<,2m}).\label{eqn:-almostepi-neg}
	\end{align}
	
	Denoting by $\lambda$ the function in \eqref{defn:lambda} by  Lemma \ref{l:thin_fourier} we rewrite the right-hand side as
	\begin{align}
	\GG_{2m}(h_{<,\alpha})&-(1+\eps)\GG_{2m}(h_{<,2m})\nonumber\\
	&=  M^2\|h_{2m}\|_{L^2(\partial B_1)}^2\frac{\eps(-\lambda(2m)+\lambda(\alpha))}{d+2\alpha-2} -\frac{\eps}{d+2\alpha-2}\sum_{\{j\,,\,\alpha_j<2m\}}^\infty (-\lambda_j+\lambda(\alpha)) c_j^2
	.\label{e:2m:energy:mainest-neg}
	\end{align}
	Since $2m-\sfrac12<\alpha<2m$, then setting $C_2:= \lambda(2m-1/2)-\lambda(2m-1) >0$, such that 
	\begin{align}
	\sum_{\{j\,,\,\alpha_j<2m\}}  (\lambda(\alpha)-\lambda_j)c_j^2&\ge C_2\sum_{\{j\,,\,\alpha_j<2m\}}c_j^2={C_2} \| Q\|_{L^2(\partial B_1)}^{2} \geq \frac{C_2}{C_1} M^2,\label{e:salviamo_le_orche-neg}
	\end{align}
	where in the last inequality we used \eqref{e:2m:gamma-neg}. Since $-\lambda(2m)+\lambda(\alpha) =\eps{(2m+\alpha+d-2)^2} $, combining \eqref{eqn:-almostepi-neg}, \eqref{e:salviamo_le_orche-neg} and \eqref{e:2m:energy:mainest-neg} we get 
	\begin{align}
	\GG_{2m}(h)-(1+\eps)\GG_{2m}(z)&\le 
	M^2\|h_{2m}\|_{L^2(\partial B_1)}^2\frac{\eps^2(2m+\alpha+d-2)^2}{d+2\alpha-2} -\frac{\eps C_2 M^2}{C_1(d+2\alpha-2)} \nonumber
	\\
	&\leq \frac{M^2\eps}{d+2\alpha-2} \Big(\|h_{2m}\|_{L^2(\partial B_1)}^2{(4m+d)^2}\eps- \frac{C_2}{C_1}  \Big).
	\label{e:2m:energy-neg}
	\end{align}
	Choosing $\eps:= \eps(d,m)$ small enough, namely $\alpha$ sufficiently close to $2m$ by the choice of $\eps$ in \eqref{e:choice_of_alpha-neg}, we find that the right-hand side in \eqref{e:2m:energy-neg} is less than or equal to $0$, that is \eqref{e:epi:2m-neg}.
\end{proof}

\subsection{Proof of Theorem \ref{t:freq}} We divide the proof in two steps.

\subsubsection{Frequencies $\frac32$ and $2m$}
	
	We first prove \eqref{e:class:3/2}. Let $c:\partial B_1\to\R$ be the trace of a $\sfrac32$-homogeneous non-trivial global solution $z\in \KK_{\sfrac32}$ of the thin-obstacle problem. Let $v$ be the competitor defined in \eqref{e:competitor:3/2}. By the optimality of $z$ and the epiperimetric inequality \eqref{e:epi:3/2} we get that 
$$0=\GG_{\sfrac32}(z)\le \GG_{\sfrac32}(v)\le \left(1-\frac1{2d+3}\right)\GG_{\sfrac32}(z)=0,$$
and in particular both the inequalities are in fact equalities. By Remark \ref{oss:3/2} we get that $z=Ch_e+r^{3/2}\phi$, where $C\ge 0$, $e\in\partial B_1'$ and $\phi:\partial B_1\to\R$ is an eigenfunction of the sperical Laplacian, corresponding to the eigenvalue $\lambda(2)=2d$, and such that $\phi\ge 0$ on $\partial B_1'$. Thus, we have 
\begin{align*}
0=\G(z)&=\G(h)+\G(r^{\sfrac32}\phi)+2C\left(\int_{B_1}\nabla h_e\cdot\nabla (r^{\sfrac32}\phi(\theta))-\frac32\int_{\partial B_1} h_e\phi\,d\HH^{d-1}\right)\\
&\ge \G(h)+\G(r^{\sfrac32}\phi)\ge \G(r^{\sfrac32}\phi)\ge 0, 
\end{align*}   
where, by Lemma \ref{l:thin_fourier} the last inequality is an equality if and only if $\phi\equiv0$. Thus, $z=Ch_{e}$ for some $e\in\partial B_1'$ and $C\ge 0$. Since $0\neq\KK_{\sfrac32}$ we get that $C>0$, which concludes the proof of \eqref{e:class:3/2}.
	
	We now prove \eqref{e:class:2m}. Suppose that $c\in H^1(\partial B_1)$ is the trace of $2m$-homogeneous non-trivial global solution of the thin-obstacle problem. Let $h$ be the competitor from \eqref{e:competitor:2m}. By the optimality of $r^{2m}c(\theta)$ and the improved version of the logarithmic epiperimetric inequality \eqref{e:epi:2m:improved} we have 
	$$0=\GG_{2m}(r^{2m}c)\le \GG_{2m}(h)\le \GG_{2m}(r^{2m}c)\big(1-\eps \,|\GG_{2m}(r^{2m}c)|^{\gamma}\big)-\eps_2\|\nabla_\theta\phi\|_{L^2(\partial B_1)}^{2+2\gamma}=-\eps_2\|\nabla_\theta\phi\|_{L^2(\partial B_1)}^{2+2\gamma}.$$
	Thus, necessarily $\|\nabla_\theta\phi\|_{L^2(\partial B_1)}=0$, that is the Fourier expansion of $c$ on the sphere $\partial B_1$ contains only low frequencies: $\ds c(\theta)=\sum_{\{j\,,\,\alpha_j\le 2m\}}c_j\phi_j(\theta).$ Now by Lemma  \ref{l:thin_fourier} we get 
	$$0=\GG_{2m}(r^{2m}c)=\frac1{4m+d-2}\sum_{\{j\,,\,\alpha_j\le 2m\}}(\lambda_j-\lambda(2m))c_j^2\le0,$$
	and so all coefficients, corresponding to frequencies with $\alpha_j<2m$, must vanish. Thus $c$ is a non-zero eigenfunction on the sphere corresponding to the eigenvalue $\lambda(2m)=2m(2m+d-2)$.

\subsubsection{Frequency gap}
	Let us first prove that 
	\begin{center}
		$\KK_{\lambda}=\emptyset$ for every $\lambda\in(\sfrac32,2)$.
	\end{center}
	Let $\lambda=\sfrac32+t\in(\sfrac32,2)$ be an admissible frequency and $c\in H^1(\partial B_1)$ a non-trivial function whose $(\sfrac32+t)$-homogeneous extension $r^{\sfrac32+t} c(\theta)\in \KK_{\sfrac32+t}$ is a solution of the thin-obstacle problem. 
	Let $v$ be the competitor from \eqref{e:competitor:3/2}. By the minimality of $r^{\sfrac32+t} c(\theta)$, Theorem \ref{t:epi:3/2} and Lemma \ref{l:mumut}, applied with $\mu=\sfrac32$, we have that 
	\begin{align*}
	\GG_{3/2}(r^{\sfrac32+t}c)\le \GG_{3/2}(v)&\le\left(1-\frac1{2d+3}\right)\GG_{3/2}(r^{\sfrac32}c)=\left(1-\frac1{2d+3}\right)\left(1+\frac{t}{d+1}\right)\GG_{3/2}(r^{\sfrac32+t}c).
	\end{align*}
	Since  $\GG_{3/2}(r^{\sfrac32+t}c)>0$, we get
	$$\left(1-\frac1{2d+3}\right)\left(1+\frac{t}{d+1}\right)\ge 1,$$ 
	which implies that $t\ge1/2$ and concludes the proof of the claim.

	We now fix $m\in\N_+$. We will show that there are constants $c_m^+>0$ and $c_m^->0$, depending only on $d$ and $m$, such that 
	\begin{center}
		 $\KK_\lambda=\emptyset$ for every $\lambda\in (2m-C_-,2m+C_+)\setminus \{2m\}$.
	\end{center} 
	Let $\lambda=2m+t$ be an admissible frequency and $c\in H^1(\partial B_1)$, $\|c\|_{L^2(B_1)}=1$, a trace whose $(2m+t)$-homogeneous extension $r^{2m+t} c(\theta)$ is a minimizer of the thin-obstacle problem. 
	
	Suppose first that $t>0$. Let $h$ be the competitor from \eqref{e:competitor:2m}. By the minimality of $r^{2m+t}c$, Theorem \ref{t:epi:2m} and Lemma \ref{l:mumut}, applied with $\mu=2m$, we have that 
	\begin{align*}
	\GG_{2m}(r^{2m+t}c)\le \GG_{2m}(h)\le\left(1-\eps t^\gamma\right) \GG_{2m}(r^{2m}c)=\left(1-\eps t^\gamma\right)\left(1+\frac{t}{4m+d-2}\right)\GG_{3/2}(r^{\sfrac32+t}c),
	\end{align*}
	where for the first inequality we used that $\GG_{2m}(r^{2m}c)\ge \GG_{2m}(r^{2m+t}c)= t>0$. By the positivity of $\GG_{2m}(r^{2m+t}c)$, we get  
	$$\left(1-\eps t^\gamma\right)\left(1+\frac{t}{4m+d-2}\right)\ge 1,$$
	which provides us with the constant $c_m^+$. 
	
	Let now $t<0$. Let $h$ be the competitor from \eqref{e:competitore_senza_nome}. By the minimality of $r^{2m+t}c$, Proposition \ref{prop:epi:2m-neg} and Lemma \ref{l:mumut}, applied with $\mu=2m$, we have  
	\begin{align*}
	\GG_{2m}(r^{2m+t}c)\le \GG_{2m}(h)\le\left(1+\eps \right) \GG_{2m}(r^{2m}c)=\left(1+\eps \right)\left(1+\frac{t}{4m+d-2}\right)\GG_{3/2}(r^{\sfrac32+t}c).
	\end{align*}
	Now since $\GG_{2m}(r^{2m+t}c)=t<0$ we get that 
	\begin{equation}
	\label{eqn:facciamoloesplicito}
	\left(1+\eps \right)\left(1+\frac{t}{4m+d-2}\right)\le1,
	\end{equation}
	which gives us $\ds c_m^-=\eps (4m+d-2)/(1+\eps)$, where $\eps$ is the constant from Proposition \ref{prop:epi:2m-neg}.
\qed

\begin{oss}\label{rmk:explicit}
Taking for instance $d=3$, $m=2$, we show how the constants in Theorem~\ref{t:freq} can be made explicit. The polynomial $h_{4}$ of Remark \ref{rmk:h2m} is given by $\frac{32}3 x_3^4-10x_3^2(x_1^2+x_2^2)+(x_1^2+x_2^2)^2$ (and $\|h_{4}\|_{L^2(\partial B_1)} \sim 9.6$), the constant $C_1=16$ in \eqref{e:2m:gamma-neg} is the number of eigenfunctions with homogeneity less than $4$, the constant $C_2 $ in \eqref{e:salviamo_le_orche-neg} is $\sfrac{15}4$.
Hence, the optimal $\eps$ in \eqref{e:2m:energy-neg} and the corresponding $c_2^-$ deduced from \eqref{eqn:facciamoloesplicito} are given by
$$\eps = \frac{C_2}{C_1\|h_{4}\|_{L^2(\partial B_1)}^2{11^2} }\qquad\text{and}\qquad c_2^- = \frac{9\eps}{1+\eps} \geq 0.0015.$$
\end{oss}

\section{Regularity of the regular and singular parts of the free-boundary}\label{sec:regFB}

The first part of Theorem \ref{t:main} was first proved in \cite{AtCaSa}. Once we have the epiperimetric inequality \eqref{e:epi:3/2}, it follows by a standard argument that can be found for example in \cite{FoSp, GaPeVe}. So we proceed with the proof of (ii). We start with the following proposition.

\subsection{Rate of convergence of the blow-up sequences}
Before starting the proof we remark that, by a simple scaling argument, if in Theorem \ref{t:epi:2m} we replace the condition \eqref{e:orchecattive} with
$$
\int_{\de B_1} c^2\,d\HH^{d-1} \leq \Theta\qquad\text{and}\qquad  |\GG_{2m}(z)|\le \Theta,
$$
for some $\Theta>0$, then the epiperimetric inequality \eqref{e:epi:2m} still holds, with $\eps$ replaced by $\eps\,\Theta^{-\gamma}$. We will use this in the first step of the proof of the following

\begin{prop}[Decay of the Weiss' energy]\label{p:decay}
	Let $u\in H^1(B_1)$ be a minimizer of $\mathcal E$. Then for every $m\in \N$ and every compact set $K\Subset B_1'\cap \mathcal S^{2m}$, there is a constant $C:=C(m,d, K, \|u\|_{H^1(B_1)})>0$ such that for every free boundary point $x_0\in \mathcal S^{2m}\cap K$, the following decay holds
	\begin{equation}\label{e:L^2decay}
	\| u_{x_0,t}-u_{x_0,s} \|_{L^1(\partial B_1)} \leq C\, (-\log(t))^{-\frac{1-\gamma}{2\gamma}} \qquad  \mbox{for all}\quad 0< s<t < \dist (K,\de B_1)\,.
	\end{equation}
	In particular the blow-up limit of $u$ at $x_0$ is unique. 
\end{prop}

\begin{proof} We divide the proof in three steps.
	
	{\it Step 1. Applicability of epiperimetric inequality at every scale.} Let $\int_{B_1}u^2\,dx=\Theta_0$. Then, by the monotonicity of $\frac{H^{x_0}(r)}{r^{d-1+2\lambda}}$, for every $\lambda$ (see Lemma \ref{lem:freq}), we deduce that 
	\begin{align*}
	\Theta_0	
		&\geq \int_{B_{R}(x_0)}u^2\,dx\geq \int_{\sfrac{R}2}^{R}H^{x_0}(r)\,dr\geq \left( \sfrac{R}2 \right)^{d-1+2\lambda} \,\frac{H^{x_0}(\sfrac{R}2)}{\sfrac{R}2},
	\end{align*}
	where $R:=\dist(x_0,\de B_1)$. 
	 In particular we have, using again the monotonicity of $\frac{H^{x_0}(r)}{r^{d-1+2\lambda}}$,
	 $$
	 0\leq \frac{H^{x_0}(r)}{r^{d-1+2\lambda}} \leq \left( \frac2{R} \right)^{d-1+2\lambda}\,	\Theta_0\,,\qquad \mbox{for every }0<r<\sfrac{R}2\mbox{ and }x_0\in B_1\,.
	 $$
	 For what concerns $\GG^{x_0}_{\lambda}$ notice that
	 $$
	 \GG_{\lambda}^{x_0}(R)\leq R^{d-2+2\lambda} \int_{B_1}|\nabla u|^2\,dx\,,\qquad \mbox{for every }x_0\in B_1\,.
	 $$
	 Taking $\lambda=2m$, it follows that for every $0<r_0\le 1$ and for every $x_0\in B_{1-r_0}$, we can apply \eqref{e:epi:2m} for every $0<r<r_0$ and every rescaling $\ds u_{x_0,r}(x)=\frac{u(x_0+rx)}{r^{2m}}$ with $\Theta$ depending on $r_0$, $d$, $m$ and $\|u\|_{H^1(B_1)}$.
	\smallskip
	
	{\it Step 2. Closeness of the blow ups for a given point $x_0$.} Let $r_0>0$ and
	$x_0 \in B_{1-r_0}$ and let $r\in (0, r_0]$. Then by Step 1 we can apply \eqref{e:epi:2m} to $\ds u_{x_0,r}$ for every $0<r<r_0$. We claim that
	$$	\| u_{x_0,t}-u_{x_0,s} \|_{L^1(\partial B_1)} \leq C\, (-\log(t/r_0))^{-\frac{1-\gamma}{2\gamma}} \qquad  \mbox{for all}\quad 0< s<t <r_0\,.
	$$
	
\noindent 	We assume $x_0=0$ without loss of generality, we fix $m\in \N$ and write $\GG(r)=\GG_{2m}^{x_0}(r,u)$. By \eqref{e:monot} 
	\begin{equation}\label{e:mat_ema}
	\frac{d}{dr}\GG(r)=\frac{(d-2+4m)}{r}\left(  \GG(z_r)-\GG(r) \right)+\underbrace{\frac1r\int_{\de B_1} \left( \nabla u_r\cdot \nu-2m\,u_r \right)^2\,d\HH^{d-1}}_{=:f(r)}
	\end{equation}
	and the epiperimetric inequality of Theorem~\ref{t:epi:2m}, there exists a radius $r_0>0$ such that for every $r \leq r_0$
	\begin{equation}
	\label{eqn:epi-applied}
	\frac{d}{dr}\GG(r) \geq \frac{d-2+4m}{r} \big( \GG(z_r)-\GG(r) \big) +f(r)\geq \frac{c}{r} \GG(r)^{1+\gamma} +2f(r)
	\end{equation}
	where $c=\eps\,\Theta^{-\gamma}(d-2+4m)$ and $\gamma \in (0,1)$ is a dimensional constant. In particular we obtain that
	\begin{equation}\label{e:mon_rem}
	\frac{d}{dr}\Big(\frac{-1}{\gamma \GG(r)^\gamma} - c \log r\Big) =\frac{1}{ \GG(r)^{1+\gamma}}\frac{d}{dr}\GG(r)- \frac{c}{r} \geq \frac{1}{ \GG(r)^{1+\gamma}} f(r) \geq 0
	\end{equation}
	and this in turn implies that $-{\GG(r)^{-\gamma}} - c\gamma \log r$ is an increasing function of $r$, namely that $\GG(r)$ decays as 
	\begin{equation}\label{eqn:e-logar}
	\GG(r) \leq ({\GG(r_0)^{-\gamma}+c \gamma \log r_0-c \gamma \log r})^{\frac {-1}{\gamma}} \leq (-c \gamma \log (r/r_0))^{\frac {-1}{\gamma}}.
	\end{equation}
	For any $0<s<t<r_0$ we estimate the $L^1$ distance between the blow-up at scales $s$ and $t$ through the Cauchy-Schwarz inequality and the monotonicity formula \eqref{e:mat_ema}
	\begin{align}\label{e:imp_1}
	\int_{\de B_1}\left| u_t-u_s \right| \,d\,\HH^{d-1}
	&\leq \int_{\de B_1}\int_s^t\frac{1}{r}\left|  x\cdot \nabla u_r-2u_r   \right| \,dr\,d\HH^{n-1} \notag\\
	&\leq \big({d \omega_d}\big)^{1/2} \int_s^t  r^{-\sfrac12} \left(\frac 1 r \int_{\de B_1} \left|  x\cdot \nabla u_r-2u_r   \right|^2   \,d\HH^{d-1} \right)^{1/2}\,dr \notag\\
	&\leq \Big({ \frac{d \omega_d}{2}} \Big)^{1/2} \int_s^t r^{-\sfrac{1}{2}}(\GG'(r))^{1/2}\, dr \\
	& \leq \Big({ \frac{d \omega_d}{2}} \Big)^{1/2} (\log(t)-\log(s))^{1/2} (\GG(t)-\GG(s))^{1/2} \notag
	\,.
	\end{align}
	Let $0<s<t<r_0/2$ and $0\le j\le i$ be such that $s/r_0\in [2^{-2^{i+1}}, 2^{-2^i})$ and $t/r_0\in [2^{-2^{j+1}}, 2^{-2^j})$. Applying the previous estimate \eqref{eqn:e-logar} to the exponentially dyadic decomposition, we obtain 
	\begin{align}\label{e:imp_2}
	\int_{\de B_1}\left| u_t-u_s \right| \,d\,\HH^{d-1}&\leq \int_{\de B_1}\left| u_t- u_{2^{-2^{j+1}}r_0} \right| \,d\,\HH^{d-1}  \notag
	\\&+ \int_{\de B_1}\left| u_{2^{-2^{i}}r_0}-u_s \right| \,d\,\HH^{d-1} +  \sum_{k=j+1}^{i-1}  \int_{\de B_1}\left| u_{2^{-2^{k+1}}r_0}-u_{2^{-2^{k}}r_0} \right| \,d\,\HH^{d-1} \notag
	\\
	&\leq C \sum_{k=j}^{i} \left(\log\big(2^{-2^{k}}\big)- \log\big(2^{-2^{k+1}}\big)\right)^{1/2}\left(\GG\big(2^{-2^{k}}r_0\big)- \GG\big(2^{-2^{k+1}}r_0\big) \right)^{1/2}\notag
	\\
	&\leq C \sum_{k=j}^{i} 2^{k/2}\GG\big(2^{-2^{k}}r_0\big)^{1/2}
	\leq C \sum_{k=j}^{i} 2^{(1-1/\gamma)k/2} \notag
	\\&\leq C 2^{(1-1/\gamma)i/2} \leq C (-\log(t/r_0))^{\frac{\gamma-1}{2\gamma}}
	\,,
	\end{align}
	where $C$ is a constant, depending on $d$, $m$, $r_0$ and $\|u\|_{H^1(B_1)}$, that may vary from line to line.
	
	\smallskip
	{\it Step 3. Conclusion.} 
	We notice that for $t \leq r_0^2$, we have $\log(t/r_0) \leq \frac12 \log t$, so that  
	$$	\| u_{x_0,t}-u_{x_0,s} \|_{L^1(\partial B_1)} \leq C\, (-\log(t))^{-\frac{1-\gamma}{2\gamma}}\qquad \mbox{for every}\quad 0<s<t<r_0^2.  $$
	Since $u_{x_0,t}$ is bounded in $L^2(\partial B_1)$ for every $t\le r_0$, by possibly enlarging the constant $C$, the above inequality holds for $0<s<t<r_0$. 
\end{proof} 
	
\subsection{Non-degeneracy of the blow-up} We now use the previous Proposition to prove that the blow-up limits are non-trivial. This is the only part of the proof of Theorem \ref{t:main} where the frequency of the point plays a role.

\begin{lemma}[Non-degeneracy]\label{lem:nondeg}
	Let $u\in H^{1}(B_1)$ be a minimizer of $\mathcal{E}$ and let $x_0\in \mathcal{S}^{\lambda}$, where $\lambda\in\{\sfrac32\}\cup\{2m\,:\,m\in\N\}$. Then the following strict lower bound holds
	$$
	H^{x_0}_0:=\lim_{r\to 0}\frac{H^{x_0}(r)}{r^{d-1+2\lambda}}>0\,.
	$$
	In particular, since by the strong $L^2(\de B_1)$ convergence of $u_{x_0,r}$ to the unique blow up $p_{x_0}$ we have $H_0:=\|p_{x_0}\|^2_{L^2(\de B_1)}$, it follows that $p_{x_0}$ is non-trivial.
\end{lemma}

\begin{proof}
	Without loss of generality we can suppose that $x_0=0$. We give the proof for $\lambda:=2m=N(0)$ for some $m\in\N$, the case $\lambda=\sfrac32$ being analogous. Assume by contradiction that 
	$$
	h(r):=\left(\frac{H(r)}{r^{d-1}}\right)^{\sfrac12}=o(r^{\lambda})
	$$
	and consider the sequence $\ds u_r(x):=\frac{u(rx)}{h(r)}$. It follows that $\|u_r\|_{L^2(\de B_1)}=1$ for every $r$, and so, by the monotonicity of the frequency function
	$$
	\int_{B_1} |\nabla u_r|^2\,dx =\frac1{r^{d-2}}D(r)\leq N(1)\, \frac1{r^{d-1}}H(r)\leq N(1)\,,
	$$
	so that, up to a not relabeled subsequence, $u_r$ converges weakly in $H^1(B_1)$ and strongly in $L^2(\partial B_1)$ to some function $p_\lambda\in H^1(B_1)$ such that $\|p_{\lambda}\|_{L^2(\de B_1)}=1$. Moreover, since $N(0)=\lambda$, $p_{\lambda}$ is a $\lambda$-homogeneous function. Notice also that due to Theorem \ref{t:ACS} the convergence is locally uniform in $B_1$. Next, for every $u_r$ consider its blow-up sequence $\ds[u_r]_\rho(x):=\rho^{-\lambda}{u_r(\rho x)}$. By Proposition \ref{p:decay}, we know that, for every $r>0$, there exists a unique blow-up limit $p_{\lambda,r}=\lim_{\rho\to0}[u_r]_\rho$. Moreover, since all the functions $u_r$ are uniformly bounded in $H^1(B_1)$, $\|u_r\|_{H^1(B_1)}^2\le N(1)+1$, there is a constant $C$ depending on the dimension, $\lambda$ and $N(1)$ such that
	\begin{equation}\label{e:poop}
	\| \ds[u_r]_t-p_{\lambda,r} \|_{L^2(\partial B_1)}^2 \leq C\, (-\log(t))^{-\frac{1-\gamma}{\gamma}} \qquad  \mbox{for all}\quad 0<t < 1\,,
	\end{equation}
	where we used the regularity of $u$ to replace the $L^1$-norm from Proposition \ref{p:decay} with the $L^2$-norm.
	Using our contradiction assumption and the strong convergence of $[u_r]_\rho$ to $p_{\lambda,r}$ in $L^2(\de B_1)$, we have
	$$
	\|p_{\lambda,r}\|_{L^2(\de B_1)}=\lim_{\rho\to 0} \frac{1}{\rho^{d-1+2\lambda}} \int_{\de B_\rho}u_r^2\,d\HH^{d-1}=\frac{r^{d-1+2\lambda}}{\int_{\de B_r}u^2\,d\HH^{d-1}}\,\lim_{\rho\to 0}\frac{1}{(r\,\rho)^{d-1+2\lambda}}\int_{\de B_{r\rho}}u^2\,d\HH^{d-1}=0
	$$ 
	for every $r>0$. It follows that, for fixed $\rho>0$ (that we will choose small enough), we have
	\begin{align*}
	1
		&=\frac1{\rho^{d-1+2\lambda}}\int_{\de B_\rho} p_\lambda^2\,d\HH^{d-1}
		\leq \frac2{\rho^{d-1+2\lambda}}\int_{\de B_\rho} |p_\lambda-u_r|^2\,d\HH^{d-1}+\frac2{\rho^{d-1+2\lambda}}\int_{\de B_\rho} u_r^2\,d\HH^{d-1}	\\
		&= \frac2{\rho^{d-1+2\lambda}}\int_{\de B_\rho} |p_\lambda-u_r|^2\,d\HH^{d-1}+2\int_{\de B_1}  [u_r]_\rho^2\,d\HH^{d-1}\\
		&\leq   \frac2{\rho^{d-1+2\lambda}}\int_{\de B_\rho} |p_\lambda-u_r|^2\,d\HH^{d-1}+C\, (-\log(\rho))^{-\frac{1-\gamma}{\gamma}} ,
	\end{align*}
	where the first equality follows from the $\lambda$-homogeneity of $p_\lambda$ and the last inequality from the rate of decay of $ [u_r]_\rho$ to $p_{\lambda,r}\equiv 0$ in \eqref{e:poop}. Choosing first $\rho>0$ and then $r=r(\rho)>0$ we reach a contradiction.
\end{proof}

\subsection{Proof of Theorem \ref{t:main}}
	Let $m\in\N_+$ be fixed and let be $x_1,x_2\in \freq^{2m}$. Let $p_{x_1}$ and $p_{x_2}$ be the unique blow-ups of $u$ at $x_1$ and $x_2$ respectively. Then we can write $p_{x_1}=\lambda_1\,p_1$ and $p_{x_2}=\lambda_2\,p_2$, where $p_1$ and $p_2$ are normalized such that $p_1,p_2\in \HH_{2m}\setminus \{0\} $. Notice that 
	\begin{equation}\label{e:0}
	\|p_{1} - p_{2}\|_{L^\infty(B_1)} \leq c(d) \int_{\partial B_1} | p_{1}(x) - p_{2}(x)| \, d\HH^{d-1}(x)\,,
	\end{equation}
	since $\|p_1\|_{L^2(\de B_1)}=1=\|p_2\|_{L^2(\de B_1)}$ and they are $2m$-homogeneous.
	
	Next notice by the triangular inequality
	$$\| p_{x_1} - p_{x_2}\|_{L^1(\partial B_1)} \leq \| u_{x_1, r} - p_{x_1}\|_{L^1(\partial B_1)} + \| u_{x_1, r}-u_{x_2, r}\|_{L^1(\partial B_1)} + \| u_{x_2, r} - p_{x_2}\|_{L^1(\partial B_1)}  
	$$
	
	Recalling that $u \in C^{1,\sfrac12}$ and that $\nabla u(x_1)= 0$, we estimate the term in the middle with
	\begin{equation}\label{e:1}
	\begin{split}
	\| u_{x_1, r}-u_{x_2, r}\|_{L^1(\partial B_1)} &\leq \int_{\partial B_1} \int_0^1  \frac{ |\nabla u (x_1+rx+ t(x_2-x_1))| |x_2-x_1|}{r^{2m}} \, dt \, d\HH^{d-1}(x)
	\\
	&\leq
	C\| u \|_{C^{1,\sfrac12}(B_r(x_1))} \frac{(r+|x_2-x_1|)^{\sfrac12}\, |x_2-x_1|}{r^{2m}}\leq	C|x_1-x_2|^{\sfrac1{8m}},
	\end{split}
	\end{equation}
	where we have set $r:= |x_1-x_2|^{\sfrac1{4m}}$. Moreover, if we assume that $r_0$ satisfies the inequality $ |r_0| (-\log |r_0|)^{-\frac{1-\gamma}{2\gamma}} \leq \dist (\{x_1,x_2\} , \partial B_1)$, then by Proposition~\ref{p:decay} we see that
	\begin{equation}\label{e:2}
	\begin{split}
	\| u_{x_1, r} - p_{x_1}\|_{L^1(\partial B_1)} + \| u_{x_2, r} - p_{x_2}\|_{L^1(\partial B_1)} 
	&\leq C \, (-\log(r))^{-\frac{1-\gamma}{2\gamma}}=  C \big(-\log |x_1-x_2|\big)^{-\frac{1-\gamma}{2\gamma}}
	\end{split}
	\end{equation}
	Putting together this inequality with \eqref{e:1} and \eqref{e:2}, we find
	\begin{equation}\label{e:3}
	\| p_{x_1} - p_{x_2}\|_{L^1(\partial B_1)} \leq  C (-\log|x_1-x_2| )^{-\frac{1-\gamma}{2\gamma}}\,.
	\end{equation}
	 
	\noindent  Next, using \eqref{e:der_H} and \eqref{eqn:e-logar} we can estimate
	 $$
	 \frac{d}{dr}\left(\frac{H^{x_i}(r)}{r^{d-1+4m}} \right)=2\frac{\GG_{2m}^{x_i}(r)}{r}\leq \frac C{r(-\gamma \log (r/r_0))^{\frac {1}{\gamma}}}\,,
	 $$
	 which integrated gives
	 \begin{equation}\label{e:decay_H}
	 \frac{H^{x_i}(t)}{t^{d-1+4m}}-\lambda_{x_i}^2\leq C\, (-\log t)^{-\frac{1-\gamma}{\gamma}}\qquad \mbox{for all }0<t<\dist(x_i,\de B_1)\,.
	 \end{equation}
	 Notice that in the previous integration we have used the fact that, by definition of $p_i$ and by the strong convergence in $L^2(\de B_1)$ of the blow ups we have
	 $$
	 \lim_{r\to 0}\frac{H^{x_i}(r)}{r^{d-1+4m}}=\|p_{x_i}\|_{L^2(\de B_1)}^2=\lambda_{x_i}^2\,.
	 $$
	Using \eqref{e:1} together with \eqref{e:decay_H}, we get
	\begin{align}\label{e:4}
	|\lambda_{x_1}-\lambda_{x_2}|^2
		&\leq C\, \left| \lambda_{x_1}^2-\frac{H^{x_1}(r)}{r^{d-1+4m}}\right| +C\,  \left|\frac{H^{x_1}(r)}{r^{d-1+4m}}- \frac{H^{x_2}(r)}{r^{d-1+4m}} \right|   +C\,  \left| \lambda_{x_2}^2-\frac{H^{x_2}(r)}{r^{d-1+4m}}\right|\notag\\
		&\leq C \,(-\log(r))^{-\frac{1-\gamma}{\gamma}}+C\, \int_{\de B_1}|u_{x_1,r}^2-u_{x_2,r}^2|\,d\HH^{d-1} \notag\\
		&\leq  C \,(-\log(r))^{-\frac{1-\gamma}{\gamma}}+C\, \| u_{x_1, r}-u_{x_2, r}\|_{L^1(\partial B_1)}^2\stackrel{\eqref{e:1}}{\leq} C (-\log|x_1-x_2| )^{-\frac{1-\gamma}{\gamma}}\,,
	\end{align}
	where the choice of $r$ is the same as above.
	
	Finally, using \eqref{e:0}, \eqref{e:3} and \eqref{e:4} we easily conclude that
	\begin{equation}\label{e:5}
	\|p_1-p_2\|_{L^\infty(B_1)} \leq C\, (-\log|x_1-x_2| )^{-\frac{1-\gamma}{\gamma}} \qquad \mbox{for every }x_1,x_2\in K\cap \freq^{2m}\Subset B_1
	\end{equation}
	where the constant $C$ depends on $m,d,\dist(K,\de B_1)$. 
	
	Now consider the collection of points $\freq^{2m}_k$, for some $m\in \N$ and $0<k<d-2$ and notice that, for every $K\Subset B_1\cap\freq^{2m}_k $, we can apply the Whitney extension theorem \cite[Whitney extension theorem]{Fefferman} to extend the function $(\tilde p_{x})_{x\in K}\subset \HH_{2m}$, where $\lambda _x \tilde p_x=p_x$ is the unique blow up at $x$ to get a function $F\in C^{2m,\log}(\R^d)$, such that $\de^\alpha F(x)=\de^\alpha \tilde p_x(0)$.  Since $x\in \freq^{2m}_k$ and the blow-ups are non-degenerate (see Lemma \ref{lem:nondeg}), there are $d-1-k$ linearly independent vectors $e_i\in \R^{d-1}$, $i=1,\dots,d-1-k$, such that 
	$$
	e_i\cdot \nabla_{x'}\tilde p_x\neq 0\qquad \mbox{on }\R^d\,. 
	$$
	It follows that there are multi-indices $\beta_i$ of order $|\beta_i|=2m-1$, such that $\de_{e_i} \de^{\beta_i}F(x)= \de_{e_i} \de^{\beta_i} \tilde p_x(0)\neq 0$. On the other hand
	$$
	\freq^{2m}\cap K=K\subset \bigcap_{i=1}^{d-1-k}\{\de^{\beta_i}F=0 \}
	$$
	so that an application of the implicit function theorem in a neighborhood of each point $x\in K$ combined with the arbitrary choice of $K$ yields that for every $x\in \freq^{2m}$ there exists $r=r(x)>0$ such that 
	$$
	\freq^{2m}_k\cap B_r(x)\quad \mbox{is contained in a $k$-dimensional $C^{1,\log}$ submanifold}\,. 
	$$
	From here the conclusion follows.\qed

\bibliographystyle{plain}
\bibliography{references-Cal}

\end{document}